\documentclass[reqno,a4paper,12pt]{amsart}%{article}%
\usepackage{amsmath,amssymb,amsthm,amsfonts,mathrsfs}
\usepackage[]{graphicx}%[dvips]
\usepackage{setspace} 
\usepackage{subfig}
\usepackage{enumitem}

\numberwithin{equation}{section}
{\theoremstyle{definition}
\newtheorem{defn}{Definition}[section]}
\newtheorem{theorem}{Theorem}[section]
\newtheorem{proposition}[theorem]{Proposition}

\newtheorem{corollary}[theorem]{Corollary}
\newtheorem{lemma}[theorem]{Lemma}

{\theoremstyle{definition}%{plain}
{%\theorembodyfont{\normalfont\rmfamily}
\newtheorem{remark}[theorem]{Remark}

}}
\newcommand{\cal}{\mathcal}

\newcommand{\BB}{{\cal B}}

\newcommand{\II}{{\cal I}}

\newcommand{\PP}{{\cal P}}

\newcommand{\RR}{{\cal R}}

\newcommand{\Nn}{{\mathbb{N}}}

\newcommand{\Pp}{{\mathbb{P}}}

\newcommand{\Rr}{{\mathbb{R}}}

\newcommand{\Zz}{{\mathbb{Z}}}

\newcommand{\Q}{{\mathbb{Q}}}
\newcommand{\Z}{{\mathbb{Z}}}

\def\diag{\operatorname{diag}}

\def\sgn{\operatorname{sgn}}

\newcommand{\conj} {\overline}

\newcommand{\id}  {\operatorname{id}}

\newcommand{\te}[1]{\quad\text{#1}\quad}
\newcommand{\comment}[1]{}

\setlist{nolistsep}

\begin{document}
\title[SRB measures for hyperbolic polygonal billiards]{SRB measures for hyperbolic polygonal billiards}
\date{\today}

\author[Del Magno]{Gianluigi Del Magno}
\address{CEMAPRE, ISEG\\
Universidade T\'ecnica de Lisboa\\
Rua do Quelhas 6, 1200-781 Lisboa, Portugal}
\email{delmagno@iseg.utl.pt}

\author[Lopes Dias]{Jo\~ao Lopes Dias}
\address{Departamento de Matem\'atica and CEMAPRE, ISEG\\
Universidade T\'ecnica de Lisboa\\
Rua do Quelhas 6, 1200-781 Lisboa, Portugal}
\email{jldias@iseg.utl.pt}

\author[Duarte]{Pedro Duarte}
\address{Departamento de Matem\'atica and CMAF \\
Faculdade de Ci\^encias\\
Universidade de Lisboa\\
Campo Grande, Edif\'icio C6, Piso 2\\
1749-016 Lisboa, Portugal}
\email{pduarte@ptmat.fc.ul.pt}

\author[Gaiv\~ao]{Jos\'e Pedro Gaiv\~ao}
\address{CEMAPRE, ISEG\\
Universidade T\'ecnica de Lisboa\\
Rua do Quelhas 6, 1200-781 Lisboa, Portugal}
\email{jpgaivao@iseg.utl.pt}

\author[Pinheiro]{Diogo Pinheiro}
\address{CEMAPRE, ISEG\\
Universidade T\'ecnica de Lisboa\\
Rua do Quelhas 6, 1200-781 Lisboa, Portugal}
\email{dpinheiro@iseg.utl.pt}

\begin{abstract}
We prove that polygonal billiards with contracting reflection laws exhibit hyperbolic attractors with countably many ergodic SRB measures. These measures are robust under small perturbations of the reflection law, and the tables for which they exist form a generic set in the space of all polygons. Specific polygonal tables are studied in detail.
\end{abstract}

\maketitle
%\tableofcontents

%%%%%%%%%%%%%%%%%%%%%%%%%%%%%%%%%%%%%%%%%%%%%%%%%%%%%%%%%%%%%%%%%%%%%%%%%%%%

\section{Introduction}
%!TEX root = dissipative_main.tex

A planar billiard is the dynamical system describing the motion of a point particle moving freely in the interior of a connected compact subset $P\subset\Rr^2$ with piecewise smooth boundary. The particle slides along straight lines until it hits the boundary $\partial P$. If the collision occurs at a smooth boundary point, then the particle gets reflected according to a prescribed rule, called the reflection law. Otherwise, the particle stops and stays forever at the non-smooth boundary point where the collision occurred. The billiard dynamics is not smooth: the singularities of the dynamics are generated by non-smooth boundary points and trajectories having tangential collisions with $ \partial P $. The reflection law most commonly considered is the specular one: the tangential component of the particle velocity remains the same, while the normal component changes its sign. Nevertheless, there exist  alternative reflection laws that are equally reasonable and have interesting mathematical implications. In this paper, we study a family of these alternative laws.

One of the interesting features of billiards is that they exhibit a broad spectrum of dynamics.
On one end of the spectrum, there are the integrable billiards. The only known integrable billiard tables are ellipses, rectangles, equilateral triangles, right isosceles triangles and right triangles with an angle $ \pi/6 $. Polygonal billiards are close to integrable systems in the sense that all their Lyapunov exponents are zero, which is a direct consequence of having zero topological entropy \cite{K87}. In particular, polygonal billiards are never hyperbolic. For some nice reviews on polygonal billiards, see \cite{G96,tabachnikov95,S00}. On the other end of the spectrum, there are the hyperbolic billiards, i.e. billiards with non-zero Lyapunov exponents. The Sinai billiards were the first examples of hyperbolic billiards; their tables consist of tori with a finite number of convex obstacles \cite{Sinai70}. Another important class of hyperbolic billiards is represented by the semi-focusing billiards, so called because their tables are formed by convex (outward) curves and possibly straight segments. The most famous examples are the Bunimovich stadium and the Wojtkowski cardioid \cite{bunimovich74,bunimovich92,W86}. All the three examples of hyperbolic billiards mentioned above have strong ergodic properties; they are Bernoulli systems, and in particular ergodic and K-mixing. Other semi-focusing hyperbolic billiards were discovered by Markarian and Donnay \cite{D91,M88}. Examples of hyperbolic billiards of great interest in statistical mechanics are the Lorentz gases and hard-ball systems \cite{sz00}. For a recent and rather complete account on the theory of hyperbolic billiards, we refer the reader to \cite{cm06}.

In this paper, we are interested in planar polygonal billiards with contracting reflection laws. This means that the function describing the dependence of the reflection angle on the incidence angle is a contraction. The simplest example of such a law is the linear contraction $\theta \mapsto \sigma \theta $ with $0 < \sigma < 1$, already studied in \cite{arroyo09,arroyo12,MDDGP12,markarian10}. To distinguish polygonal billiards with contracting reflection laws from those with specular reflection law, we will refer to the former as contracting polygonal billiards and to the latter as standard polygonal billiards.

The dynamics of these two classes of billiards are very different. Standard polygonal billiards exhibit the characteristic features of parabolic systems: zero entropy, subexponential growth of periodic orbits and subexponential divergence of nearby orbits. Instead, as we show in this paper, contracting polygonal billiards are typically uniformly hyperbolic systems (with singularities), and so they have positive entropy, exponential divergence of orbits, dense periodic orbits, and countably many ergodic and mixing components. Moreover, contracting polygonal billiards may exhibit mixed behavior, with hyperbolic and non-hyperbolic attractors coexisting. Yet another essential difference between these two classes of billiards is that standard polygonal billiards preserve the phase space area, whereas the contracting ones do not, this being due to the phase space contraction produced by the new reflection law.
Therefore, an important part of the analysis of contracting polygonal billiards is represented by the study of their invariant measures. In this paper, we give sufficient conditions for the existence of Sinai-Ruelle-Bowen measures (SRB measures for short), which are measures with non-zero Lyapunov exponents and absolutely conditional measures on unstable manifolds.
Our results rely on a general theory of SRB measures for uniformly systems with singularities developed by Pesin~\cite{Pesin92} and Sataev~\cite{Sataev92}.

The first results on billiards with contracting reflection laws were obtained by Markarian, Pujals and Sambarino. In \cite{markarian10}, they proved that many billiards with contracting reflection law, including polygonal billiards, have dominated splitting, i.e. the tangent bundle of their phase space splits into two invariant directions such that the growth rate along one direction dominates uniformly the growth rate along the other direction. Note that uniform hyperbolicity implies dominated splitting. Markarian, Pujals and Sambarino also obtained some results concerning the hyperbolicity of polygonal billiards with strongly contracting reflection laws. In this paper, we improve their results. Specific polygonal tables and linear contracting reflection laws were studied in \cite{arroyo12} and \cite{MDDGP12}. In the first of these two papers, the existence of a unique SRB measure was established for the equilateral triangle using a natural Markov partition. In the second one, we proved the existence of horseshoes and several properties of hyperbolic periodic orbits for the square. Works on billiards related to ours are the one of Zhang, where she proved the existence of SRB measures and obtained Green-Kubo formulae for two-dimensional periodic Lorentz gases with reflection laws modified according to certain `twisting' rules \cite{Zhang11}, and the one of Chernov, Karepanov and Sim\'anyi on 
a gas of hard disks moving in a two-dimensional cylinder with non-elastic collisions between the disks and the walls \cite{CKS12}. 
Finally, see also \cite{Altmann08,arroyo09}, for numerical results on billiards with non-specular reflections laws relevant in the study of certain optical systems.

We now describe the organization of this paper and at the same time our results in more detail.
In Section~\ref{sec:preliminaries}, we give the definitions of the billiard map with general reflection law and of a contracting reflection law. We then discuss basic properties of the billiard map, including its singular sets, derivative and attractors. 

The general hyperbolic properties of contracting polygonal billiards are discussed in Section~\ref{sec:hyperbolicity}. There, we prove that every contracting polygonal billiard has a dominated splitting, and provide conditions for the existence of hyperbolic sets. As a corollary, we obtain that if a polygon does not have parallel sides, then the correspondent contracting billiard has hyperbolic attractors.

Sufficient conditions for the existence of SRB measures on hyperbolic attractors are given in Section~\ref{sec:SRB}. As already mentioned, the result we obtain relies on general results of Pesin and Sataev. To be able to use these results, we show that contracting polygonal billiards satisfy a crucial condition, which roughly speaking, amounts to requiring that billiard trajectories do not visit too often a small neighborhood of the vertexes of the polygon. In the same section, we also prove that the existence of an SRB measure is a robust property under small perturbation of the reflection law.

In Section~\ref{sec:Generic}, we prove that polygonal billiards with strongly contracting reflection laws have generically countably many ergodic SRB measures. We then obtain the same conclusion for every regular polygon with an odd number of sides in Section~\ref{sec:regpoly}. 
To arrive to these results, we use a `perturbation argument'. Indeed, for strongly contracting reflection laws, the billiard map is close to the one-dimensional map obtained as the limit of the billiard map when the contraction becomes infinitely large. This map is called the slap map, and was introduced in \cite{markarian10}. If a polygon does not have parallel sides, then the slap map is expanding, and so the billiard map remains hyperbolic provided that the reflection law is strongly contracting. 
In Section~\ref{sec:regpoly}, we also derive some general results about the absence of hyperbolic attractors when the billiard table is a regular polygon with an even number of sides. 

Section~\ref{sec:triangles} is devoted to the proof of the existence of SRB measures for acute triangles. The simpler dynamics of this class of polygons allows us to obtain an explicit estimate on the strength of the contraction needed for the existence of SRB measures.
Finally, in Section~\ref{sec:rectangles}, we give sufficient conditions for the existence and absence of hyperbolic attractors in rectangular billiards.

\section{Preliminaries}
\label{sec:preliminaries}
%!TEX root = dissipative_main.tex

In this section, we define the mathematical objects we intend to study. First, we give a detailed construction of the billiard map for polygonal tables and general reflection laws. Then, we introduce a class of contracting reflection laws, and for such laws, we derive a series of preliminary results concerning the singular sets, the derivative and the attracting sets of polygonal billiard maps.

\subsection{Billiard map}
Let $ P $ be a non self-intersecting $ n $-gon. Denote by $ |\partial P| $ the length of $ \partial P $ and by $ C_{1},\ldots,C_{n} $ the vertexes of $ P $. Let $ \zeta \colon [0,|\partial P|] \to \Rr^{2} $ be the positively oriented parametrization  of the curve $ \partial P $ by arclength such that if $ \tilde{s}_{i} $ is the arclength parameter corresponding to $ C_{i} $, then $ 0 = \tilde{s}_{1} < \cdots < \tilde{s}_{n} < \tilde{s}_{n+1}= |\partial P| $. Note that the parametrization $ \zeta $ is smooth everywhere on $ [0,|\partial P|] $ except at $ \tilde{s}_{2},\ldots,\tilde{s}_{n} $, and that $ \zeta(0) = \zeta(|\partial P|) $, because $ \partial P $ is a closed curve.

Consider the following subsets of $ \Rr^{2} $, 
\[ 
M = (0,|\partial P|) \times \left(-\frac{\pi}{2},\frac{\pi}{2}\right)
\quad \text{and} \quad 
V = \{\tilde{s}_{1},\ldots,\tilde{s}_{n+1}\} \times \left(-\frac{\pi}{2},\frac{\pi}{2}\right).
\]
For every $ x = (s,\theta) \in M \setminus V $, define $ q(x) = \zeta(s) $, and let $ v(x) $ be the unit vector of $ \Rr^{2} $ such that the oriented angle formed by $ \zeta'(s) $ and $ v(x) $ is equal to $ \pi/2-\theta $. Thus, $ (q(x),v(x)) $ defines a unit tangent vector of $ \Rr^{2} $ pointing inside $ P $. Both pairs $ (s,\theta) $ and $ (q(x),v(x)) $ specify the state of the billiard particle immediately after a collision with $ \partial P $. For this reason, we will refer to both $ (s,\theta) $ and $ (q(x),v(x)) $ as the \emph{collisions} of the billiard particle.

Given $ x \in M \setminus V $ and $ \tau>0 $, denote by $ (q(x),q(x)+\tau v(x)) $ the open segment of $ \Rr^{2} $ with endpoints $ q(x) $ and $ q(x)+\tau v(x) $. Next, define 
\[ 
A(x) = \{\tau>0 \colon (q(x),q(x)+\tau v(x)) \subset P \}.
\]
This set is nonempty and bounded, and so we can define
\[
t(x) = \sup A(x) \quad \text{and} \quad q_{1}(x) = q(x) + t(x) v(x).
\]
It is easy to see that $ q_{1}(x) $ is the point where the particle leaving the point $ q(x) $ with velocity $ v(x) $ hits $ \partial P $, whereas $ t(x) $ is the euclidean distance between $ q(x) $ and $ q_{1}(x) $.

Let $ S^{+}_{1} $ be the closure in $ \Rr^{2} $ of all collisions mapped by $ q_{1} $ into vertexes of $ P $, i.e., 
\[ 
S^{+}_{1} = \bigcup^{n}_{i=1} \overline{q^{-1}_{1}(C_{i})}.
\] 
Also, let 
\[ 
N^{+}_{1} = \partial M \cup V \cup S^{+}_{1}.
\]

For every $ x \in M \setminus N^{+}_{1} $, define 
\[ 
s_{1}(x) = \zeta^{-1}(q_{1}(x)) 
\]
and
\[
\bar{v}_{1}(x) = -v(x) + 2\langle \zeta'(s_{1}(x)),v(x) \rangle \zeta'(s_{1}(x)),
\] 
where $ \langle \cdot\, , \cdot \rangle $ is the euclidean scalar product of $ \Rr^{2} $. The billiard particle leaving $ q(x) $ in the direction $ v(x) $ hits $ \partial P $ at $ q_{1}(x) $, and undergoes a reflection. When this reflection is specular, i.e., the reflection and the incidence angles coincide, the velocity of the particle immediately after the collision is equal to $ \bar{v}_{1}(x) $. The angle formed by $ v_{1}(x) $ with $ \zeta'(s_{1}(x)) $ is given
\[ 
\bar{\theta}_{1}(x) = \arcsin (\langle \zeta'(s_{1}(x)),\bar{v}_{1}(x)\rangle).
\]

Let $ L_{0} $ and $ L_{1} $ be the sides of $ P $ containing $ q(x) $ and $ q_{1}(x) $, respectively. Also, let $ \ell_{0} $ and $ \ell_{1} $ be the lines through the origin of $ \Rr^{2} $ parallel to $ L_{0} $ and $ L_{1} $, respectively. We assume that $ \ell_{0} $ and $ \ell_{1} $ have the same orientation of $ L_{0} $ and $ L_{1} $. Denote by $ \delta(L_{0},L_{1}) $ the smallest positive angle of the rotation that maps $ \ell_{0} $ to $ \ell_{1} $ counterclockwise. We call such an angle the \emph{angle formed by the sides $ L_{0} $ and $ L_{1} $}. It is not difficult to deduce that the relation between $ \bar{\theta}_{1}(x) $ and $ \theta(x) $ read as
\begin{equation}
\label{eq:angle-after}
\bar{\theta}_{1}(x) =  \pi - \delta(L_{0},L_{1}) - \theta(x).
\end{equation}

The billiard map $ \bar{\Phi} \colon M \setminus N^{+}_{1} \to M $ with the specular reflection law is given by 
\[
\bar{\Phi}(x) = (s_{1}(x),\bar{\theta}_{1}(x)) \qquad \text{for } x \in M \setminus N^{+}_{1}.
\]
This map is a smooth embedding \cite[Theorem~2.33]{cm06}. 

\subsection{General reflection laws} 
We now give the definition of the billiard map with general reflection law. A reflection law is given by a function $ f \colon (-\pi/2,\pi/2) \to (-\pi/2,\pi/2) $. If we define $ \theta_{1}(x) = f(\bar{\theta}_{1}(x)) $ for $ x \in M \setminus N^{+}_{1} $, then the billiard map $ \Phi_{f} \colon M \setminus N^{+}_{1} \to M $ with reflection law $ f $ (see Fig.~\ref{fi:polygon}) is given by 
\[ 
\Phi_{f}(x) = (s_{1}(x),\theta_{1}(x)) \qquad \text{for } x \in M \setminus N^{+}_{1}. 
\]
Let $ R_{f}(s,\theta)=(s,f(\theta)) $ for every $ (s,\theta) \in M $. So $ \Phi_{f} $ may be written as
\[
\Phi_{f} = R_{f} \circ \bar{\Phi}.
\]
It follows that $ \Phi_{f} $ is a $ C^{k} $ embedding with $ k \ge 1 $ provided that $ f $ is.
% a $ C^{k} $ embedding with $ k \ge 1 $ as well.
\begin{figure}
  \begin{center}
    \includegraphics[width=5cm]{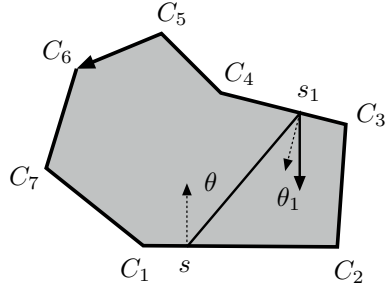}
  \end{center}
  \caption{Billiard map $ \Phi_{f} $.}
  \label{fi:polygon}
\end{figure}

\subsection{Contracting reflection laws}
Given a differentiable function $ f \colon (-\pi/2,\pi/2) \to \Rr $, let
$$
\lambda(f) = \sup_{\theta \in (-\pi/2,\pi/2)} |f'(\theta)|.
$$
Denote by $ \RR $ the set of all differentiable $ f \colon (-\pi/2,\pi/2) \to \Rr $ such that $ f(0)=0 $ and $ \lambda(f)<\infty $. For every transformation $ \Psi \colon M \setminus N^{+}_{1} \to \Rr^{2} $, define 
\[
\|\Psi\|_{\infty} = \sup_{x \in M \setminus N^{+}_{1}} \|\Psi(x)\|_{2},
\]
where $ \|\cdot\|_{2} $ is the euclidean norm of $ \Rr^{2} $. Let $ X $ be the set of all $ \Psi \colon M \setminus N^{+}_{1} \to \Rr^{2} $ such that $ \|\Psi\|_{\infty} < \infty $. The spaces $ (\RR,\lambda) $ and $ (X,\|\cdot\|_{\infty}) $ are normed space. 

Consider the map from $ \RR $ to $ X $ defined by $ f \mapsto \Phi_{f} $. This map is continuous. Indeed, if $ f_{1} $ and $ f_{2} $ belong to $ \RR $, then 
\begin{align*}
\|\Phi_{f_{1}}-\Phi_{f_{2}}\|_{\infty} & \le \|R_{f_{1}}-R_{f_{2}}\|_{\infty} \\
& \le \sup_{\theta \in (-\pi/2,\pi/2)} |f_{1}(\theta)-f_{2}(\theta)| \\
& \le \frac{\pi}{2} \lambda(f_{1}-f_{2}).
\end{align*}

For each $ k \ge 1 $, we define $\RR^k_1$ to be the set of all $ f \in \RR $ that are $ C^{k} $ embeddings with $ \lambda(f) < 1 $. If $ f \in \RR^{k}_{1} $, then $ f $ is a strict contraction, and $ \theta=0 $ is its unique fixed point. Since $ f $ is monotone, it admits a continuous extension up to $ -\pi/2 $ and $ \pi/2 $, which is naturally denoted by $ f(-\pi/2) $ and $ f(\pi/2) $. 

Strictly increasing reflection laws $ f \in \RR^{k}_{1} $ were also considered in \cite{markarian10}. The simplest example of such a law is $ f(\theta) = \sigma \theta $ with $ 0 < \sigma < 1 $ \cite{arroyo09,MDDGP12,markarian10}. See Fig.~\ref{fi:reflection}.

From now on, we assume that $f\in\RR^k_1$ with $ k \ge 1 $ unless otherwise stated.

\begin{figure}
  \begin{center}
    \includegraphics[width=4cm]{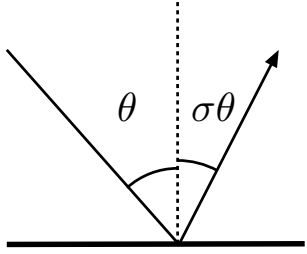}
  \end{center}
  \caption{Contracting reflection law}
  \label{fi:reflection}
\end{figure}

\subsection{Singular sets}
%\begin{defn}
	A curve $ \gamma $ contained in $ M $ is called \emph{strictly decreasing} (resp. \emph{strictly increasing}), if $ \gamma $ is the graph of a continuous strictly decreasing (resp. \emph{strictly increasing}) function $ h \colon I \to (-\pi/2,\pi/2) $   with $ I $ being an interval of $ (0,|\partial Q|) $. 
	%A curve $ \gamma $ contained in $ M $ is called \emph{horizontal} if $ \gamma = J \times \{a\} $ with $ J $ being an interval of $ (0,|\partial Q|) $ and $ a \in (-\pi/2,\pi/2) $. Finally, a curve $ \gamma $ contained in $ M $ is called \emph{vertical} if $ \gamma = \{b\} \times K $ with $ b \in (0,|\partial Q|) $ and $ K $ being an interval of $ (-\pi/2,\pi/2) $.
%\end{defn}

In the next two propositions, we state the main geometrical properties of the sets $ S^{+}_{1} $ and $ S^{-}_{1} $. Their proofs are given in the Appendix.
%~\ref{ap:first}.

\begin{proposition}
	\label{pr:prop-sing}
	The set $ S^{+}_{1} $ is a union of finitely many analytic compact curves $ \Gamma_{1},\ldots,\Gamma_{m} $ such that 
	\begin{enumerate} 
		\item $ \Gamma_{i} \subset [\tilde{s}_{k},\tilde{s}_{k+1}] \times (-\pi/2,\pi/2) $ for some $ 1\le k \le n $,
		\item $ \Gamma_{i} $ is strictly decreasing,
		\item $ \Gamma_{i} $ and $ \Gamma_{j} $ intersect transversally for $ i \neq j $,
		\item $ \Gamma_{i} \cap \Gamma_{j} \subset \partial \Gamma_{i} \cup \partial \Gamma_{j} $ for $ i \neq j $.
	\end{enumerate}	
\end{proposition}

Let $ F(s,\theta) = (s,f(-\theta)) $ for $ (s,\theta) \in M $, and define
\[
S^{-}_{1} = F(S^{+}_{1}).
\]
Since $ F $ is a $ C^{k} $ embedding, $ S^{+}_{1} $ and $ S^{-}_{1} $ are diffeomorphic.

\begin{proposition}
	\label{pr:singminus}
	The conclusions of Proposition~\ref{pr:prop-sing} hold for $ S^{-}_{1} $ if $ f'<0 $, and hold with `decreasing' replaced by `increasing' in conclusion~(3) if $ f'>0 $.
\end{proposition}

The iterates of the sets of $ S^{+}_{1} $ and $ S^{-}_{1} $ play a crucial role in the study of the properties of the billiard map. Hence, for $ n \ge 1 $, define recursively 
\[ 
S^{+}_{n+1} =  S^{+}_{n} \cup \Phi^{-1}_{f}(S^{+}_{n}) \qquad \text{and} \qquad S^{-}_{n+1} = S^{-}_{n} \cup \Phi_{f}(S^{-}_{n}).
\]

Let $ N^{-}_{1} = \partial (\Phi_{f}(M \setminus N^{+}_{1})) $. Note that while the set $ N^{+}_{1} $ does not depend on the reflection law $ f $, the set $ N^{-1}_{1} $ does. 
For later use, we also introduce the iterates of $ N^{+}_{1} $ and $ N^{-}_{1} $. For $ n \ge 1 $, set
\[ 
N^{+}_{n+1} =  N^{+}_{n} \cup \Phi^{-1}_{f}(N^{+}_{n}) \qquad \text{and} \qquad N^{-}_{n+1} = N^{-}_{n} \cup \Phi_{f}(N^{-}_{n}).
\]
% It is not difficult to see that 
% \[
% N^{\pm}_{n} = V \cup S^{\pm}_{n}. 
% \]
The set $ N^{+}_{n} $ (resp. $ N^{-}_{n} $) consist of the points of $ M $ where the map $ \Phi^{n}_{f} $ (resp. $ \Phi^{-n}_{f} $) is not defined, and is called the \emph{singular set} of $ \Phi^{n}_{f} $ (resp. $ \Phi^{-n}_{f} $). Finally, let $ N^{+}_{\infty} = \bigcup_{n \ge 1} N^{+}_{n} $, and let $ N^{-}_{\infty} = \bigcup_{n \ge 1} N^{-}_{n} $.

\subsection{Derivative of the billiard map} 
\label{sec: derivative bill map}
Given $ (s_{0},\theta_{0}) \in M \setminus N^{+}_{1} $, let $ (s_{1},\theta_{1}) = \Phi_{f}(s_{0},\theta_{0}) $. Recall that $ t(s_{0},\theta_{0}) $ denotes the length of the segment connecting $ q(s_{0},\theta_{0}) $ and $ q_{1}(s_{0},\theta_{0}) $. For the computation of the differential of $ \bar{\Phi} $, see for instance \cite[Formula~(2.26)]{cm06}. Since $ D \Phi_{f} = D R_{f} \cdot D \bar{\Phi} $, we easily obtain
\begin{equation*}
	\begin{split}
		 D \Phi_{f}(s_{0},\theta_{0})
		 = - \begin{pmatrix}
				\dfrac{\cos \theta_{0}}{\cos \bar{\theta}_{1}} & \dfrac{t(s_{0},\theta_{0})}{\cos \bar{\theta}_{1}} \\[1em]
				0 & f'(\bar{\theta}_{1}) \\
			\end{pmatrix}.
	\end{split}
\end{equation*}

Now, suppose that $ (s_{0},\theta_{0}) \in M \setminus N^{+}_{n} $ for some $ n>0 $. Let $ (s_{i},\theta_{i}) = \Phi_{f}^{i} (s_{0},\theta_{0}) $ for $ i = 1,\ldots,n $. We easily see that
\begin{equation}
	\label{eq:dern}
	D\Phi_{f}^{n} (s_{0},\theta_{0}) = (-1)^{n}
	 \begin{pmatrix}
			\alpha_{n}(s_{0},\theta_{0})  & \gamma_{n}(s_{0},\theta_{0}) \\[1em]
			0 & \beta_{n}(s_{0},\theta_{0}) \\
	\end{pmatrix},
\end{equation}
where
\[
\alpha_{n}(s_{0},\theta_{0}) = \frac{\cos \theta_{0}}{\cos \bar{\theta}_{n}} \prod^{n-1}_{i=1} \frac{\cos \theta_{i}}{\cos \bar{\theta}_{i}}, \qquad \beta_{n}(s_{0},\theta_{0}) = \prod^{n}_{i=1} f'(\bar{\theta}_{i}),
\]
and
\[ 
\gamma_{n}(s_{0},\theta_{0}) = \sum^{n-1}_{i=0} \beta_{i}(s_{0},\theta_{0}) \frac{t(s_{i},\theta_{i})}{\cos \bar{\theta}_{n}} \prod^{n-1}_{k=i+1} \frac{\cos \theta_{k}}{\cos \bar{\theta}_{k}}.
\] 

Let $ \rho \colon (-\pi/2,\pi,2) \to [1,+\infty) $ be the function defined by 
\[ 
\rho(\theta) =  \frac{\cos (f(\theta))}{\cos \theta}
%\frac{\cos \lambda \theta}{\cos \theta}}  
\qquad \text{for } \theta \in (-\pi/2,\pi,2).
\] 
Since $ f \in \mathcal{R}^{k}_{1} $, it follows easily that $ \rho $ is continuous, and that $ \rho(\theta) \ge 1 $ with equality if and only if $ \theta = 0 $. Also, from $ \lim_{\theta \to \pm \pi/2} \rho(\theta) = +\infty $, we obtain 
\[ 
r(\epsilon) := \min_{\epsilon \le |\theta| < \pi/2} \rho(\theta) > 1 \qquad \text{for } 0<\epsilon<\pi/2.
\]

Define
\[
\Lambda_{n-1}(s_{1},\theta_{1}) = \prod^{n-1}_{i=1} \rho(\bar{\theta}_{i}).
\]
From the definition of $ \alpha_{n} $, we see that
\begin{equation}
	\label{eq:alpha-lambda}
\alpha_{n}(s_{0},\theta_{0}) = \frac{\cos \theta_{0}}{\cos \bar{\theta}_{n}} \cdot \Lambda_{n-1}(s_{1},\theta_{1}).
\end{equation}

\subsection{Attractors}
Let
\[ 
D_{f} = \bigcap_{n \ge 0} \Phi^{n}_{f} (M \setminus N^{+}_{\infty}).
\] 
It can be checked that $ \Phi^{-1}_{f}(D_{f}) = D_{f} $, which in turn implies that $ D_{f} \cap N^{-}_{\infty} = \emptyset $. Thus, every element of $ D_{f} $ has infinite positive and negative semi-orbits. We say that a subset $ \Sigma $ of $ M $ is \emph{invariant} (resp. \emph{forward invariant}) if $ \Sigma \subset D_{f} $ and $ \Phi^{-1}_{f}(\Sigma) = \Sigma $ (resp. $ \Phi_{f}(\Sigma) \subset \Sigma $). %Hence, $ D_{f} $ is the largest invariant subset of $ M $. 
Following the terminology introduced by Pesin in his work on maps with singularities \cite{Pesin92}, we call the closure of $ D_{f} $ the \emph{attractor} of $ \Phi_{f} $.

\begin{lemma}
	\label{le:bounded-angle}
	We have 
	\[ 
	D_{f} \subset \tilde{D}_{f}:= \left\{(s,\theta) \in M: |\theta| < \frac{\pi}{2}\lambda(f) \right\}.
	\]
\end{lemma}

\begin{proof}
The conclusion is a direct consequence of the invariance of $ D_{f} $, 
and the fact that the absolute value of the angle of every element of $ \Phi_{f}(D_{f}) $ is not larger than $  \pi \lambda(f)/2 $.
\end{proof}

\begin{defn}
	We say that two sides $ L_{1} $ and $ L_{2} $ of $ P $ \emph{see each other} if there exists a line segment of a billiard trajectory whose endpoints lie in the interior of $ L_{1} $ and $ L_{2} $. 
\end{defn}

Let $ \nu $ be the Lebesgue measure $ ds d\theta $ of $ \Rr^{2} $. When the reflection law is the standard one, i.e., $ f = 1 $, we have $ D_{f} = M \setminus (N^{-}_{\infty} \cup N^{+}_{\infty}) $, and $ \cos \theta d s d \theta $ is the natural invariant measure for $ \Phi_{f} $. In this case, it is easy to see that $ \nu(D_{f}) = \nu(M) >0 $. For $ f \in \mathcal{R}^{k}_{1} $, the situation is more delicate; we may have $ \nu(D_{f})=0 $, and therefore only singular invariant measures for $ \Phi_{f} $. The measures that play a prominent role in the study of hyperbolic systems are the SRB measures. We will prove the existence of these measures in Section~\ref{sec:SRB}. In the following propositions, we give sufficient conditions for $ \nu(D_{f}) = 0 $. 

\begin{proposition}
	\label{pr:zero1}
Suppose that $ P $ has the property 
\[ 
|\pi - \delta(L_{i},L_{j})| \neq \pi/2
\] 
for all the pairs of sides $ L_{i} $ and $ L_{j} $ seeing each other. Then $ \nu(D_{f}) = 0 $ for $ \lambda(f) $ sufficiently small.
\end{proposition}

\begin{proposition}
	\label{pr:zero2}
Suppose that 
\[ 
\sup_{\theta \in (-\frac{\pi}{2},\frac{\pi}{2})} \frac{|f'(\theta)|}{\cos \theta} < 1.
\]
Then $ \nu(D_{f}) = 0 $.
\end{proposition}

\begin{remark}
	An example of a reflection law that satisfies the hypothesis of Proposition~\ref{pr:zero2} is $ f(\theta) = \sigma \sin \theta $ with $ 0 < \sigma < 1 $. 
	%$|f'(|\theta|)| = O(\pi/2 - |\theta|) $ as $ |\theta| \to (\pi/2)^{-} $. 
\end{remark}

\subsection{Notation}
Instead of $ \Phi_{f}, \lambda(f),D_{f} $, we write $ \Phi, \lambda, D $ when no confusion can arise. We also write $ x_{n} = (s_{n},\theta_{n}) $ for $ \Phi^{n}_{f} (x_{0}) $ with $ x_{0}= (s_{0},\theta_{0}) $.

\section{Hyperbolicity} 
\label{sec:hyperbolicity}
%!TEX root = dissipative_main.tex

Let $ u \in T_{x} M $ with $ x \in M $, and denote by $ u_{s} $ and $ u_{\theta} $ the components of $ u $ with respect to the coordinates $ (s,\theta) $. By setting $ \|u\|^{2} = u^{2}_{s} + u^{2}_{\theta} $ for every $ u \in T_{x}M $, we introduce a differentiable norm $ \| \cdot \| $ on $ M $. Let $ d $ be the metric on $ M $ generated by $ \| \cdot \| $.

We say that an invariant set $ \Sigma \subset M $ has a \emph{dominated splitting} if there exist a non-trivial continuous splitting $ T_{\Sigma} M = E \oplus F $ and two constants $ 0 < \mu < 1 $ and $ A > 0 $ such that for all $x \in \Sigma$ and $n \ge 1$, we have
\[
\|D\Phi^{n}|_{E(x)}\| \cdot \|D\Phi^{-n}|_{F(\Phi^{n} x)}\| \le A \mu^{n}.
\]

A invariant set $ \Sigma \subset M $ is \emph{hyperbolic} if there exist a non-trivial measurable splitting $ T_{\Sigma} M = E \oplus F $ and two measurable functions $ 0 < \mu < 1 $ and $ A > 0 $ on $\Sigma$ such that for all $ x \in \Sigma$ and $ n \ge 1$, we have
\begin{equation*}
\begin{split}
\|D\Phi^{n}|_{E(x)}\| & \le A(x) \mu(x)^{n} \\
\|D\Phi^{-n}|_{F(\Phi^{n} x)}\| & \le A(x) \mu(x)^{n}.
\end{split}
\end{equation*}
If the functions $\mu$ and $A$ can be replaced by constants and the splitting is continuous, then $\Sigma$ is called \emph{uniformly hyperbolic}, otherwise it is called \emph{non-uniformly hyperbolic}. If a set is uniform hyperbolicity, then clearly it admits a dominated splitting.

\subsection{Dominated splitting}
\label{su:dominated}
In \cite{markarian10}, Markarian, Pujals and Sambarino proved that for every contracting reflection law with $ f'>0 $ and a large family of planar domains, including polygons, the invariant sets of $ \Phi $ have dominated splitting. Their proof is based on the construction of an invariant cone field. We give below a proof of this fact only restricted to polygonal billiards but including reflection laws with $ f'<0 $ that is much more direct than the proof in \cite{markarian10}, not requiring the construction of an invariant cone field.

\begin{proposition}
	\label{pr:dominated}
	Every invariant set $ \Sigma $ of $ \Phi $ has a dominated splitting. Moreover, 
	\begin{enumerate}
		\item if $ \limsup_{n \to +\infty} n^{-1} \log \alpha_{n}(x) > 0 $ for every $ x \in \Sigma $, then $ \Sigma $ is hyperbolic,
		\item if there exist $ A>0 $ and $ \mu>1 $ such that $ \alpha_{n}(x) \ge A \mu^{n} $ for every $ x \in \Sigma $, then $ \Sigma $ is uniformly hyperbolic.
	\end{enumerate}
\end{proposition}

\begin{proof}
Suppose that $ \Sigma $ is an invariant set. Given a point $ x_{0} \in \Sigma $, let $ x_{n} = \Phi(x_{0}) $. From \eqref{eq:dern}, we obtain 
\begin{equation*}
		D\Phi^{-n} (x_{n}) = (-1)^{n}
		 \begin{pmatrix}
				\alpha^{-1}_{n}(x_{0})  & -\dfrac{\gamma_{n}(x_{0})}{\alpha_{n}(x_{0}) \beta_{n}(x_{0})} \\[1em]
				0 & \beta^{-1}_{n}(x_{0}) \\
		\end{pmatrix} 
		\qquad \text{for } n>0.
	\end{equation*}

Let $ F = \{F(x_{0})\}_{x_{0} \in \Sigma} $ be the horizontal subbundle on $ \Sigma $ given by $ F(x_{0}) = \{u \in T_{x_{0}}M: u_{\theta} = 0 \} $. Clearly $ F $ is continuous and invariant. Also, we have 
\begin{equation}
	\label{eq:F-expansion}
\|D\Phi^{-n}|_{F(x_{n})}\| = \alpha^{-1}_{n}(x_{0}) \qquad \text{for } n > 0.
\end{equation}

Next, let $ V(x_{0}) = \{u \in T_{x_{0}}M:u_{s} = 0 \} $ be the vertical subspace at $ x_{0} \in \Sigma $. We define the subbundle $ E $ as the collection $ \{E(x_{0})\}_{x_{0} \in \Sigma} $ with $ E(x_{0}) $ being the subspace of $ T_{x_{0}}M $ defined by
	\begin{equation}
		\label{eq:dominated}
	E(x_{0}) = \lim_{n \to +\infty} D \Phi^{-n}(x_{n}) V(x_{n}),
	\end{equation}
	where the limit has to be understood in the projective sense. Since 
	\begin{equation*}
	D\Phi^{-n}(x_{n}) 
	\begin{pmatrix}
		0 \\
		1
	\end{pmatrix}
	= -(\beta_{n}(x_{0}))^{-n}
	\begin{pmatrix}
		-\frac{\gamma_{n}(x_{0})}{\alpha_{n}(x_{0})} \\
		1
	\end{pmatrix},
	\end{equation*}
	the existence of the limit in \eqref{eq:dominated} follows from the convergence of the sequence $ \{\gamma_{n}(x_{0}) / \alpha_{n}(x_{0})\}_{n \ge 1} $. As a matter of fact, this sequence is uniformly absolutely-convergent on $ \Sigma $. Indeed, 
	\[
	\lim_{n \to +\infty} \frac{\gamma_{n}(x_{0})}{\alpha_{n}(x_{0})} = \frac{1}{\cos \theta_{0}} \sum^{\infty}_{i=0} \frac{t(x_{i})}{\Lambda_{i}(x_{1})} \beta_{i}(x_{i}),
	\]
	and since $ t(x_{i}) \le \operatorname{diam} P $ and $ 1 \le \Lambda_{i}(x_{1}) $ for every $ i $, $ |\theta_{0}| < \lambda \pi/2 $ by Lemma~\ref{le:bounded-angle}, and $ |\beta_{n}| \le \lambda^{n} $, we have 
	\[
	\frac{1}{\cos \theta_{0}} \sum^{\infty}_{i=0} \frac{t(x_{i})}{\Lambda_{i}(x_{1})} |\beta_{i}(x_{i})| \le 
	\frac{1}{\cos \theta_{0}} \sum^{\infty}_{i=0} \frac{t(x_{i})}{\Lambda_{i}(x_{1})} \lambda^{i} \\
	\le \left(\cos \frac{\pi \lambda}{2}\right)^{-1} \frac{\operatorname{diam} P}{1-\lambda}. 
	\]
	From the convergence of $ \{\gamma_{n}(x_{0}) / \alpha_{n}(x_{0})\}_{n \ge 1} $, we can immediately deduce that the subspaces $ E(x_{0}) $ and $ F(x_{0}) $ are transversal, and the subbundle $ E $ is invariant (since $ E(x_{0}) $ is defined as a limit). Moreover, from the uniform absolute convergence of $ \{\gamma_{n}(x_{0}) / \alpha_{n}(x_{0})\}_{n \ge 1} $ and the continuity of functions $ \gamma_{n}(x_{0}) / \alpha_{n}(x_{0}) $ on $ \Sigma $, we deduce that $ E $ is continuous on $ \Sigma $.

Now, it is not difficult to see that 
	\begin{align*}
	\|D\Phi^{n}|_{E(x_{0})}\| & = \lim_{k\to +\infty} \frac {\left\|D \Phi^{n-k}(x_{k}) (0,1)^T\right\|}{\left\| D \Phi^{-k}(x_{k}) (0,1)^T \right\|} \\
	& = \lambda^n \lim_{k\to +\infty} \frac{\left(1+\left(\frac{\gamma_{k-n}(x_{n})}{\alpha_{k-n}(x_{n})}\right)^2\right)^{\frac{1}{2}}} {\left(1+\left(\frac{\gamma_{k}(x_{0})}{\alpha_{k}(x_{0})}\right)^2\right)^{\frac{1}{2}}} \\
	& \le \lambda^{n} \left(1+ \left(\cos \frac{\pi \lambda}{2}\right)^{-1} \frac{\operatorname{diam} P}{1-\lambda}\right).
	\end{align*}
	It follows that $ \|D\Phi^{n}|_{E(x_{0})}\| \le A \lambda^{n} $ for some constant $ A>0 $ independent of $ x_{0} $ and $ n $. This means that the subbundle $ E $ is uniformly contracting. Now, using \eqref{eq:F-expansion} and the fact that $ 1/\alpha_{n}(x_{0}) $ is uniformly bounded by $ 1/\cos (\pi \lambda/2) $, we can conclude that
	\begin{equation*}
	\|D\Phi^{n}|_{E(x_{0})}\| \cdot \|D\Phi^{-n}|_{F(x_{n})}\| \le  \frac{A}{\cos(\pi \lambda/2)} \lambda^n \qquad \text{for } x_{0} \in \Sigma \text{ and } n \ge 1,
	\end{equation*}
i.e., the invariant set $ \Sigma $ has dominated splitting.

Finally, to prove (1) and (2), we just need to observe that the subbundle $ E $ is uniformly contracting, and that under the hypotheses of (1) and (2), the subbundle $ F $ is expanding and uniformly expanding, respectively.
\end{proof}

\subsection{Periodic orbits} 
\label{s:perorbits}

\begin{defn}
	We say that $ K \subset M $ is an \emph{attracting set} if $ K $ has a neighborhood $ U $ in $ M $ such that $ d(\Phi^{n}(x),K) \to 0 $ as $ n \to +\infty $ for every $ x \in U \setminus N^{+}_{\infty} $. 
\end{defn}

\begin{defn}
	We denote by $ \mathcal{P} $ the set of all periodic points of period two of the billiard map $ \Phi $. 
\end{defn}

\begin{proposition}
	\label{prop:periodic-points}
	The periodic points of $ \Phi $ have the following properties: 
	\begin{enumerate}
		\item every periodic point of period two is parabolic, 
		\item the set $ \PP $ is an attracting set,
		\item every periodic point of period greater than two is hyperbolic.
	\end{enumerate}
\end{proposition}

\begin{proof}
	Suppose that $ (s_{0},\theta_{0}) $ is a periodic point of $ \Phi $ of period $ n $. Since $ (s_{n},\theta_{n}) =  (s_{0},\theta_{0}) $, it follows immediately from \eqref{eq:alpha-lambda} that $ \alpha_{n}(s_{0},\theta_{0}) = \Lambda_{n}(s_{0},\theta_{0}) $. If we define $ \zeta = \max \{\rho(\bar{\theta}_{i}):i=1,\ldots,n\} $, then we see that $ \zeta \le \alpha_{n}(s_{0},\theta_{0}) \le \zeta^{n} $.

	Since a periodic point $ (s_{0},\theta_{0}) $ of period two corresponds to a billiard trajectory that hits perpendicularly two sides of $ P $, we have $ \theta_{0} = \theta_{1} = \theta_{2} = 0 $. Hence $ \zeta = 1 $, and so $ \alpha_{2}(s_{0},\theta_{0}) = 1 $. As $ (-1)^{n} \alpha_{n} $ is an eigenvalue of $ D \Phi^{n} $ (cf. \eqref{eq:dern}), all periodic points of period two are parabolic. Using the fact that $ f $ is a contraction, it is not difficult to show that the set of all periodic points of period two is an attracting set. 

Finally, we show that periodic points of period $ n > 2 $ are hyperbolic. These points have at least two collisions with a non-zero angle, and so $ \rho(\bar{\theta}_{i}) > 1 $ for at least two $ i $'s. If we set $ \mu = \zeta^{1/n} > 1 $, then $ \alpha_{n}(s_{0},\theta_{0}) > \mu^{n} $. The wanted result now follows from part~(2) of Proposition~\ref{pr:dominated}.
\end{proof}

\subsection{Uniform hyperbolicity}
\label{su:uni-hyp}
We now address the problem of the hyperbolicity of general invariant sets. Proposition~\ref{pr:delta-zero} below represents a considerable improvement of a result of Markarian, Pujals and Sambarino, stating that for any convex polygon without parallel sides, the map $ \Phi_{f} $ is hyperbolic provided that $ \lambda(f) $ is sufficiently close to zero \cite[Corollary 4 and Theorem 23]{markarian10}.

\begin{defn}
	\label{de:bounded}
We say that a forward invariant set $ \Omega \subset M $ has Property~(A) if there exists $ m>0 $ such that every sequence of consecutive collisions between parallel sides of $ P $ contained in $ \Omega $ consists of no more than $ m $ collisions.
\end{defn}

\begin{proposition}
	\label{pr:delta-zero}
	Suppose that $ \Sigma $ is an invariant set with Property~(A). Then $ \Sigma $ is uniformly hyperbolic.
\end{proposition}

\begin{proof}
Let $ \Delta = \min_{ij} \left|\pi - \delta(L_{i},L_{j})\right|/2 $, where the minimum is taken over all pairs of non-parallel sides $ L_{i} $ and $ L_{j} $ seeing each other. It is easy to check that $ \Delta > 0 $. 
	
Let $ \{(s_{n},\theta_{n})\}_{n \in \Zz} $ be an orbit contained in $ \Sigma $. Consider two consecutive collisions $ (s_{i},\theta_{i}) $ and $ (s_{i+1},\theta_{i+1}) $ at non-parallel sides of $ P $ . By \eqref{eq:angle-after}, we have 
\begin{equation*}
\bar{\theta}_{i+1} = \pi - \delta - \theta_{i}, 
\end{equation*}
where $ \delta $ is the angle formed by the sides of $ P $ containing $ s_{i} $ and $ s_{i+1} $. 
Fix $ 0 < \epsilon < \Delta $. If $ |\theta_{i}| \ge \epsilon $, then $ |\bar{\theta}_{i}| > \epsilon $ so that 
$ \rho(\bar{\theta}_{i}) \ge r(\epsilon) $. 
On the other hand, if $ |\theta_{i}| < \epsilon $, then $ |\bar{\theta}_{i+1}| > \pi - \delta - \epsilon > 2 \Delta - \epsilon > \epsilon $, and so $ \rho(\bar{\theta}_{i+1}) \ge r(\epsilon) $. Since $ \rho \ge 1 $, in both cases, we have
\begin{equation}
	\label{eq:pair}
 \rho(\bar{\theta}_{i}) \rho(\bar{\theta}_{i+1}) \ge r(\epsilon).
\end{equation}

By Property~(A), the maximum number of consecutive collisions occurring at parallel sides of $ P $ must be bounded above by $ m > 0 $. If $ P $ does not have pair of parallel sides, then we set $ m = 1 $. Now, consider a sequence of $ m+2 $ consecutive collisions $ (s_{j},\theta_{j}),\ldots,(s_{j+m+1},\theta_{j+m+1}) $. It is clear that such a sequence contain at least one pair of consecutive collisions occurring at sides that are not parallel. Hence, by using \eqref{eq:pair} and $ \rho \ge 1 $, we deduce that 
\[ \prod^{m+1}_{k=0} \rho(\bar{\theta}_{j+k}) \ge r(\epsilon).
\]
From this inequality, we easily obtain $ \Lambda_{n}(s_{1},\theta_{1}) \ge r(\epsilon)^{\frac{n-m-1}{m+2}} $ for all $ n \ge m+2 $. But $ \Lambda_{n} \ge 1 $ for every $ n \ge 1 $ and so
\begin{equation}
	\label{eq:lambda-uniform}
\Lambda_{n}(s_{1},\theta_{1}) \ge r(\epsilon)^{\frac{n-m-1}{m+2}} \qquad \text{for } n \ge 1.
\end{equation}
Define $ A = r(\epsilon)^{-1} \cos (\pi \lambda / 2) $ and $ \mu = r(\epsilon)^{1/(m+2)} $. Combining \eqref{eq:alpha-lambda} and \eqref{eq:lambda-uniform}, and using Lemma \ref{le:bounded-angle}, we obtain
\begin{equation*}
	\alpha_{n}(s_{0},\theta_{0}) \ge A \mu^{n} \qquad \text{for } n \ge 1.
\end{equation*}
Therefore, $ \Phi $ is uniformly expanding along the horizontal direction. By Part~(2) of Proposition \ref{pr:dominated}, it follows that $ \Sigma $ is uniformly hyperbolic.
\end{proof}

\begin{corollary}
	\label{co:delta-positive}
	Suppose that $ \PP = \emptyset $. Then $ D $ is uniformly hyperbolic.
\end{corollary}

\begin{proof}
Since $ \mathcal{P} = \emptyset $, we can apply Proposition~\ref{pr:delta-zero} to $ D $.
\end{proof}

All regular $ (2n+1) $-gons satisfy the hypothesis of Corollary~\ref{co:delta-positive}.

\begin{corollary}
	\label{co:odd-polygons}
	The set $ D $ is uniformly hyperbolic for every regular $ (2n+1) $-gon. 
\end{corollary}

\begin{remark}
A necessary condition for $ \mathcal{P} \neq \emptyset $ is that $ P $ has a pair of parallel sides. This condition is clearly not sufficient. Fig.~\ref{starbilliard} displays an example of a polygon with parallel sides satisfying the hypothesis of Corollary~\ref{co:delta-positive}.
\end{remark}

\begin{figure}
  \begin{center}
    \includegraphics[width=2in]{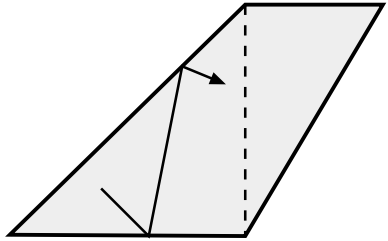}
  \end{center}
  \caption{A uniformly hyperbolic polygonal billiard with parallel sides.}
  \label{starbilliard}
\end{figure}

Even when $ \PP \neq \emptyset $, the map $ \Phi $ may still have uniformly hyperbolic sets. In Section~\ref{sec:rectangles}, we will show that this is the case for billiards in rectangles with proper contracting reflection laws (see Proposition~\ref{pr:uh-rect}).

\section{SRB measures for hyperbolic attractors}
\label{sec:SRB}
%!TEX root = dissipative_main.tex

In this section we find criteria for the existence of SRB measures in polygonal billiards with a contracting reflection law. This is a consequence of a general result of Pesin that states the existence of such measures for piecewise smooth systems satisfying certain technical conditions. We show that those holds for the case of the billiards consider.

%%%%%%%%%%%%%%%%%%%%%%%%%%%

Recall that an invariant probability measure $ \mu $ of a diffeomorphism $g\colon U\to U$ on a Riemannian manifold $ U $ is called an \emph{SRB measure} (after Sinai, Ruelle and Bowen) if it is hyperbolic, and its conditional measures on unstable manifolds are absolutely continuous.
The relevance of such a measure is due to the fact that if $\mu$ is also ergodic, then there exists a positive Lebesgue measure set of points $ x \in U $ such that
$$
\lim_{n\to+\infty}\frac1n\sum_{i=0}^{n-1}\varphi(g^i(x))=\int_U\varphi\,d\mu
$$
for every continuous function $\varphi\colon U \to \Rr$. For the setting considered in this section, this property is proved in \cite[Theorem~3]{Pesin92}. We refer the reader to~\cite{Young02}, for a nice review on SRB measures.

%%%%%%%%%%%%%%%%%%%%%%%%%%%
\subsection{Piecewise smooth maps}\label{subsect pw smooth maps}

Let $U$ be a smooth Riemannian surface with metric $ \rho $ and area $ \nu $. Let $K$ be an open connected subset of $ U $ with compact closure. 

\begin{defn}
	We say that $ g $ is a \emph{piecewise smooth map} on $ K $
if there exist four constants $a_1,a_2\geq0$, $c_1,c_2>0$ and finitely many disjoint connected open sets $ \{K_{i}\} $ with boundary formed by a finite number of $ C^{1} $ curves such that $ \overline{K} = \bigcup_{i} \overline{K_{i}} $, $ g \colon \bigcup_{i} K_{i} \to U $  is injective, $ g|_{K_{i}} $ is a $ C^{2} $ diffeomorphism on its image satisfying
\begin{equation}\label{cdn H2}
\begin{split}
\|D^2g(x)\|
&\leq 
c_1 \rho(x,N)^{-a_1},
\quad
x\in K\setminus N, \\
\|D^2g^{-1}(x)\|
&\leq 
c_2 \rho(x,N')^{-a_2},
\quad
x\in K\setminus N',
\end{split}
\end{equation}
where $ N = \bigcup_{i} \partial K_{i} $ and $ N' = \partial(g(K \setminus N)) $.
\end{defn}

Next, we define
\begin{equation*}
K^+ =\{x\in K \colon g^n(x)\not\in N,n\geq0\},
\end{equation*}
and consider the invariant set 
$$
D={\bigcap_{n \ge 0} g^{n}(K^{+})}.
$$
We call the closure $A=\conj D$ an \emph{attractor} following the nomenclature of \cite{Pesin92}.
If $D$ is uniformly hyperbolic, then we refer to $A$ as a generalized hyperbolic attractor, or simply \emph{hyperbolic attractor}. The set $A$ is said to be \emph{regular} if 
$$
D_l^-=\bigcup_{0<\delta\leq1}\left\{x\in A\colon\rho(g^{-n}(x),N')\geq \delta e^{-n l},n\geq0\right\}
$$
is not empty for all sufficiently small $l>0$.
In particular, $A$ is regular if it contains a periodic orbit.
Moreover,~\cite[Proposition 3]{Pesin92} states that if there are $c>0$ and $q>0$ such that for every $n\geq0$ and every $\varepsilon>0$,
\begin{equation}\label{suf cdn reg}
\nu\left( K^+ \cap g^{-n}(N_\varepsilon)  \right) \leq c\,\varepsilon^{q},
\end{equation}
then $A$ is regular. Here $N_\varepsilon$ denotes the $\varepsilon$-neighborhood of $N$.

If $A$ is a regular hyperbolic attractor, then there is $l>0$ such that any $ x \in D^{-}:=D^{-}_{l} $
has local unstable manifold $W^u_{loc}(x)$ ~\cite[Proposition 4]{Pesin92}. The next theorem %due to Pesin 
provides sufficient conditions for the existence of SRB measures supported on $ A $.

\begin{theorem}[\cite{Pesin92}]\label{Pesin criterium}
Suppose that
\begin{enumerate}
\item
$A$ is a regular hyperbolic attractor, 
\item
there exist $x\in D^{-}$ and $c,q,\varepsilon_0>0$ such that for $n\geq0$ and $0<\varepsilon<\varepsilon_0$, we have
\begin{equation}\label{Pesin cdn 3}
\ell\left( W_{loc}^u(x) \cap g^{-n}(N_\varepsilon)  \right) \leq c\,\varepsilon^{q},
\end{equation}
where $\ell$ is the length of a curve in $U$.
\end{enumerate}
Then $A$ has countably many ergodic SRB measures supported on $ D $.
\end{theorem}
	
Under more restrictive conditions, which for polygonal billiard maps follow from Proposition~\ref{thm14 Pesin} below,
Sataev proved that the number of ergodic SRB measures is actually finite~\cite{Sataev92}. These measures have strong ergodic properties: for every ergodic SRB measure $ \mu $, there exist $ n>0 $ and a $ g^{n}$-invariant subset $ E $ with $ \mu(E)>0 $ such that $ (g^{n}|_{E},\mu|_{E}) $ is Bernoulli and has exponential decay of correlations, and the Central Limit Theorem holds for $ (g^{n}|_{E},\mu|_{E}) $ \cite{Afraimovich95,cz09,Young98}. Moreover, the Young dimension formula, relating the Hausdorff dimension of $ \mu $ to its Lyapunov exponents and entropy, holds true \cite{Sch}.
Notice also that under the conditions of the above theorem, the hyperbolic points are dense in the attractor~\cite[Theorem~11]{Pesin92}.

Known examples of the above systems are Lozi maps, Belykh maps and geometrical Lorenz maps \cite{ap87,b80,l78,y85}. We will show in the following that polygonal billiards with strongly contracting reflection laws are also of this type.

%%%%%%%%%%%%%%%%%%%%%%%%%%%%%%%%%%%%%%%%%%%%%
\subsection{SRB measures for polygonal billiards}
\label{sec: polygonal billiards}

Let $K=M$ for the polygonal billiard map $\Phi_f$ where $f$ is any reflection law, and $N=N_1^+$. We recall that $ N^{+}_{1} = \partial M \cup V \cup S^{+}_{1} $.
From Proposition~\ref{pr:prop-sing}, it follows that the set $ N^{+}_{1} $ divides $ M $ into pairwise disjoint open connected subsets $ M_{1},\ldots,M_{k} $, i.e., $ M \setminus N^{+}_{1} = \bigcup_{i} M_{i} $.
We denote by $ \Phi_{f,i} $ the restriction $ \Phi_f|_{M_{i}} $.

The map $\Phi_f$ for a contracting reflection law $f\in\RR_1^k$, $k\geq1$, is a $C^k$-diffeomorphism from $M\setminus N_1^+$ onto its image in $M$.

The set of points that can be infinitely iterated is $K^+=M\setminus N_\infty^+$ and let
$$
A=\conj{\bigcap_{n\geq0}\Phi_f^n(K^+)}.
$$
Notice that $A$ contains for example the period two orbits, if they exist. 

We say that a polygon has \emph{parallel sides facing each other} if there is a straight line joining orthogonally two of its sides. This is the same as the billiard map having a period two orbit (i.e. $\PP\not=\emptyset$).

\begin{proposition}\label{prop gen hyp att}
$A$ is a hyperbolic attractor iff the polygon does not have parallel sides facing each other.
\end{proposition}

\begin{proof}
$A$ is a hyperbolic attractor iff it does not contain the parabolic orbits that are the period two orbits  (Proposition~\ref{pr:delta-zero}). 
These exist iff there are parallel sides facing each other.
\end{proof}

%%%%%%%%
%\subsection{Existence of SRB measures}
In the following we will prove the conditions of Theorem~\ref{Pesin criterium} for the billiard map $\Phi_f$.

A horizontal curve $\Gamma$ is a curve lying on a line of constant $\theta$ inside $ M_i$ for some $i$.
It is parametrized by a smooth path 
\begin{equation}\label{parametrization of Gamma}
\gamma\colon I\to M_i
\end{equation}
on an interval $I$ contained in $[0,1]$, satisfying $\gamma'(t)=(1,0)$, $t\in I$. 
So, $\ell(\Gamma)=|I|\leq 1$.

Given $n\geq1$, let $p(S_n^+)$ be the maximum number of smooth components of $S_n^+$ intersecting at one point.
Moreover, denote the smallest expansion rate of $\Phi_f^n$ along the unstable direction by 
$$
\alpha(\Phi_f^n)=\inf\limits_{x\in M\setminus N_n^+}\left\|D\Phi_f^n(x)\,(1,0) \right\|.
$$
Notice that $\alpha(\Phi_f^n)\geq \alpha(\Phi_f)^n$ and that $f\mapsto\alpha(\Phi_f)$ is a continuous map because $N_1^+$ does not depend on $f$.

To apply Theorem~\ref{Pesin criterium} to polygonal billiards one needs to prove ~\eqref{suf cdn reg} and ~\eqref{Pesin cdn 3}.
Both conditions follow from the conclusion~\eqref{H4} of Proposition~\ref{thm14 Pesin} below.
This proposition is the adaptation of~\cite[Theorem~14]{Pesin92} (see also ~\cite[Theorem~6.1]{Sch} and ~\cite[Section~7]{Young98}) to polygonal billiards.
This adaptation is necessary because billiards do not satisfy all conditions of the mentioned results (e.g. existence of smooth extensions) and verify trivially others (like bounded distortion estimates).

\begin{proposition}\label{thm14 Pesin}
If $m\geq1$ and
\begin{equation}\label{Property P}
p(S_m^+)<\alpha(\Phi_f^m),
\end{equation}
then there exist $c_m,\varepsilon_0>0$ such that for any horizontal curve $\Gamma \subset \tilde{D} $, $n\geq0$ and $0<\varepsilon<\varepsilon_0$,
\begin{equation}\label{H4}
\ell\left( \Gamma \cap \Phi_f^{-n}(N_\varepsilon)  \right) \leq c_m\,\varepsilon.
\end{equation}
\end{proposition}

\begin{proof}
Notice that by considering $\varepsilon_0<(1-\lambda)\pi/2$, the orbit of any point never enters the $\varepsilon_0$-neighborhood of $\partial M$. So $ N_{\varepsilon} $ in  Condition~\eqref{H4} can be safely replaced by the $ \varepsilon $-neighborhood of $ V \cup S^{+}_{1} $. Accordingly, throughout this proof, the symbols $ N^{+}_{1} $ and $ N_{\varepsilon} $ have to be understood as $ V \cup S^{+}_{1} $ and the $ \varepsilon $-neighborhood of $ V \cup S^{+}_{1} $, respectively.

We have that $G_r=\Gamma\cap N_r^+$ is finite for any $r\geq0$, where we set $N_0^+=\emptyset$. 
Indeed, after one iteration of the billiard map only a finite number of points in the horizontal curve $\Gamma$ reach a corner, and the same after $r$ iterations for each of the components of the image of $\Gamma$.
So, $\ell(\Gamma)=\ell(\Gamma\setminus G_r)\leq1$ and
\begin{equation*}
\begin{split}
\ell\left( \Gamma \cap \Phi_f^{-r}(N_\varepsilon)  \right) 
&= 
\ell\left( (\Gamma\setminus G_r) \cap \Phi_f^{-r}(N_\varepsilon)  \right) \\
&=\ell \left( \Phi_f^{-r} \left( \Phi_f^r(\Gamma\setminus G_r) \cap N_\varepsilon  \right)\right).
\end{split}
\end{equation*}

The $r$-th iterate of $\Gamma\setminus G_r$ is the union of a finite number of smooth curves (each one diffeomorphic to a component of $\Gamma\setminus G_r$) given by
$$
\Phi_f^r(\Gamma\setminus G_r)=\bigcup_{i\in J_r}\Gamma_{r,i},
$$
where $J_r$ is finite.
We have that $\Gamma_{r,i}=\Phi_{f}(\Gamma_{r-1,j}\cap M_l)$ for some $j,l$ depending on $i,r$ and $j\in J_{r-1}$. Notice that the set $J_{r}$ does not have more elements than $J_{r+1}$, and that $\Gamma_{0,i}=\Gamma$ where $i\in J_0=\{1\}$.

Fix $n=sm+s'$ where $s\geq0$ and $s'\in\{0,\dots,m-1\}$.
Since all $\widetilde\Gamma_{n,i}=\Gamma_{n,i}\cap N_\varepsilon$ are disjoint,
\begin{equation*}
\ell\left(\Gamma\cap \Phi_f^{-n}(N_\varepsilon) \right)
=
\ell\left(\Phi_f^{-sm}\circ \Phi_f^{-s'}\left(\bigcup_{i\in J_n} \widetilde\Gamma_{n,i}\right) \right)
=
\sum_{i\in J_{n}}\ell\left(\Phi_f^{-sm}(\widetilde\Gamma'_{sm,i})\right) 
\end{equation*}
where $\widetilde\Gamma'_{sm,i}=\Phi_f^{-s'}(\widetilde\Gamma_{n,i})$.

From the transversality between $N_1^+$ and the horizontal direction given by Proposition~\ref{pr:prop-sing},
there is $C'>0$ (independent of $n$ and $\Gamma$) satisfying the inequality
$
\ell(\widetilde\Gamma_{n,i}) \leq C'\varepsilon.
$
Let $\eta$ be the maximum number of intersections between a horizontal line in the phase space $M$ and $N_1^+$.
Given $j\in J_{sm}$ denote by $J_{n,j}$ the set of all $i\in J_n$ satisfying $\widetilde\Gamma'_{sm,i}\subset \Gamma_{sm,j}$.
Thus, $\# J_{n,j}\leq (\eta+1)^{s'}$.
Notice that $\cup_{j\in J_{sm}}J_{n,j}=J_n$.
Furthermore, 
\begin{equation}\label{estimate Lambda km}
\begin{split}
\sum_{i\in J_{n,j}} \ell(\widetilde\Gamma'_{sm,i})
\leq 
\sum_{i\in J_{n,j}} \frac{\ell(\widetilde\Gamma_{n,i})}{\alpha(\Phi_f^{s'})} 
\leq
\frac{(\eta+1)^{m} C'\,\varepsilon}{\min\limits_{s''<m}\alpha(\Phi_f^{s''})}
=C'_m\varepsilon.
\end{split}
\end{equation}

Notice that the definition of $p=p(S_m^+)$ is equivalent to saying that it is the smallest positive integer for which there is $C_m>0$ such that if $\Gamma$ is a horizontal curve with $\ell(\Gamma)<C_m$, then $\Gamma\setminus S_m^+$ has at most $p$ components with positive length.

Let $d=\min\{C_m,\ell(\Gamma)\}$.
Write $\psi=\Phi_f^{m}$ and $\Upsilon_{r,i}=\Gamma_{rm,i}$ with $i\in J_r$ and $0\leq r\leq s$. 

Take any $0\leq k\leq s$ and $l\in J_{km}$ such that $\ell(\Upsilon_{k,l})\geq d$, denote the set of these pairs $(k,l)$ by $\II_{s}$ and the set of indices 
$$
I(k,l)=\{i\in J_{sm}\colon \psi^{-(s-k)}(\Upsilon_{s,i})\subset \Upsilon_{k,l}\}.
$$
Now, define
$$
I_{k,l}=I(k,l)\setminus\bigcup_{(r,j)\in\II_s,r>k}I(r,j).
$$
It is now easy to check that the sets $I_{k,l}$ are disjoint, their union is $J_{sm}$ and $\#I_{k,l}\leq p^{s-k-1}$.
Therefore,
$$
J_n=\bigcup_{(k,l)\in \II_{s}}\bigcup_{j\in I_{k,l}}J_{n,j}.
$$

For each $(k,l)\in\II_{s}$ define
$$
\Lambda_{k,l}=
\bigcup_{j\in I_{k,l}} 
\bigcup_{i\in J_{n,j}}
\psi^{-(s-k)}(\widetilde\Gamma'_{sm,i}) 
\subset
\Upsilon_{k,l}.
$$
Notice that this is a disjoint union, hence
\begin{equation*}
\begin{split}
\sum_{i\in J_{n}} \ell\left(\Phi_f^{-sm}(\widetilde\Gamma'_{sm,i})\right)
&\leq
\sum_{(k,l)\in\II_{s}}
\sum_{j\in I_{k,l}}
\sum_{i\in{J_{n,j}}} 
\ell\left(\psi^{-s}(\widetilde\Gamma'_{sm,i})\right) \\
&=
\sum_{(k,l)\in\II_{s}} \ell\left(\psi^{-k}(\Lambda_{k,l})\right).
\end{split}
\end{equation*}

Each $\Upsilon_{k,l}$ is parametrized by the path $\gamma_{k,l}(t)=\psi^{k}\circ\gamma(t)$ with $\gamma_{k,l}'(t)= D\psi^{k}\circ\gamma(t)\, (1,0)$ and
\begin{equation*}
\ell(\Upsilon_{k,l}) = \int_{\gamma_{k,l}^{-1}(\Upsilon_{k,l}) } \|\gamma_{k,l}'(t)\|\, dt  =
\|\gamma_{k,l}'(\xi)\| \, \ell\left(\psi^{-k}(\Upsilon_{k,l}) \right),
\end{equation*}
where $\xi\in \gamma_{k,l}^{-1}(\Upsilon_{k,l}) \subset[0,1]$.
Similarly,
$
\ell(\Lambda_{k,l}) =
\|\gamma_{k,l}'(\xi')\| \, \ell\left(\psi^{-k}(\Lambda_{k,l})\right),
$
for some $\xi'\in \gamma_{k,l}^{-1}(\Lambda_{k,l})$.
So,
\begin{equation}\label{length relation}
\ell\left(\psi^{-k}(\Lambda_{k,l})\right) = 
\ell(\Lambda_{k,l}) \,
\frac{\|\gamma_{k,l}'(\xi)\|} {\|\gamma_{k,l}'(\xi')\|} \,
\frac{\ell\left(\psi^{-k}(\Upsilon_{k,l})\right)}{\ell( \Upsilon_{k,l})}.
\end{equation}
Following Section~\ref{sec: derivative bill map}, we have that $\|D\psi^{k}(x)\,(1,0)\|$ is constant in $\psi^{-k}(\Upsilon_{k,l})\subset\Gamma$. This proves that
$
\|\gamma_{k,l}'(\xi)\|= \|\gamma_{k,l}'(\xi')\|.
$
Moreover, from~\eqref{estimate Lambda km}, writing $\alpha=\alpha(\psi)$ which is larger than $p$ by~\eqref{Property P},
$$
\ell(\Lambda_{k,l}) 
\leq 
\sum_{j\in I_{k,l}} 
\sum_{i\in J_{n,j}} 
\frac{\ell(\widetilde\Gamma'_{sm,i})}{\alpha^{s-k}}
\leq C_m'\varepsilon \left(\frac{p}\alpha\right)^{s-k}.
$$

Finally, from $\ell(\Upsilon_{k,l})\geq d$ when $(k,l)\in\II_{s}$, using~\eqref{length relation},
$$
\sum_{(k,l)\in\II_{s}} \ell\left(\psi^{-k}(\Lambda_{k,l})\right)
\leq
\frac{C_m'\varepsilon}{d} \sum_{k=0}^s\left(\frac{p}\alpha\right)^{s-k} 
\sum_l \ell\left(\psi^{-k}(\Upsilon_{k,l})\right).
$$
Furthermore, $\sum_l \ell\left(\psi^{-k}(\Upsilon_{k,l})\right)\leq \ell\left(\psi^{-k}(\cup_l\Upsilon_{k,l})\right)=\ell(\Gamma)$.

All the above estimates imply that
$$
\ell\left(\Gamma \cap \Phi_f^{-n}(N_\varepsilon)\right)
\leq
\frac{C_m'\alpha\ell(\Gamma)}{d(\alpha-p)}\,\varepsilon.
$$
The fact that $d^{-1}\ell(\Gamma)$ is equal to 1 if $\ell(\Gamma)\leq C_m$ and otherwise bounded from above by $C_m^{-1}$ because $\ell(\Gamma)\leq 1$, proves~\eqref{H4}.
\end{proof}

%%%%%%%%%%%%%%%%%%%

Denote by $\BB$ the subset of contractive $C^2$ reflection laws $f$ in $\RR_1^2$ such that $f$ and its inverse $f^{-1}$ have bounded second derivatives.

\begin{theorem}\label{thm bill p<a}
Consider a polygon without parallel sides facing each other and suppose that $f\in\BB$.
If there is $m\geq1$ such that
$$
p(S_m^+)<\alpha(\Phi_f^m),
$$
then $A$ has countably many ergodic SRB measures.
\end{theorem}

\begin{proof}	
First, notice that by Proposition~\ref{prop gen hyp att} $A$ is a generalized hyperbolic attractor.
Moreover, the billiard map $\Phi_f$ satisfies conditions~\eqref{cdn H2}. This is a consequence of the following facts: $\Phi_f=R_f\circ \bar\Phi$, the conservative billiard map $\bar\Phi$ verifies conditions~\eqref{cdn H2} with $a_1=a_2=3$ (see~\cite{KS86}) and $R_f$ has bounded second derivatives. 

Now, the existence of countably many ergodic SRB measures follows from Theorem~\ref{Pesin criterium} provided that~\eqref{suf cdn reg} and~\eqref{Pesin cdn 3} hold. Both conditions follow from \eqref{H4}; the second one directly, whereas the first one after decomposing $K^+$ into horizontal curves.  
\end{proof}

\begin{remark}
In fact, \eqref{H4} allows us to use a result by Sataev~\cite[Theorem~5.15]{Sataev92}, and conclude that the number of ergodic SRB measures is finite.
\end{remark}

\begin{remark}
The conditional measures of the SRB measures in Theorem~\ref{thm bill p<a} coincide up to a normalizing constant factor with the measure induced by the Riemannian metric on the unstable manifolds. This can be deduced easily from the computations in the proof of \cite[Proposition~7]{Pesin92}. In particular, note that the function $ \kappa $ in that proof is identically equal to one for our billiards.
\end{remark}

%%%%%%%%%%%%%%%%%
\subsection{Sufficient conditions}

A convex polygonal billiard corresponds to $p(S_1^+)=2$, because in this case $S_1^+$ is a disjoint union of smooth curves. 
A non-convex polygonal billiard yields instead $p(S_1^+)\geq2$.

Every $\Phi_{f,i}$ admits a continuous extension map to $\conj M_i$ denoted by $\hat\Phi_{f,i}$, which can be multi-valued at a finite number of points (these correspond to trajectories of the billiard flow tangent to the sides of the polygon). 
This allows us to extend $\Phi_{f}$ to $N_1^+$ in the following multi-valued way,
\begin{equation}\label{notation gn1}
\Phi_{f}^n(B)=\bigcup_i \hat\Phi_{f,i}\left(\Phi_{f}^{n-1}(B)\cap \overline{M_i}\right),
\quad
B\subset M, 
\quad
n\geq1,
\end{equation}
and $\Phi_{f}^0=\id$ on $M$.
Similarly, for the pre-images,
\begin{equation}\label{notation gn2}
\Phi_{f}^{-n}(B)=\bigcup_i \hat\Phi_{f,i}^{-1}\left(\Phi_{f}^{-(n-1)}(B)\right).
\end{equation}

\begin{lemma}\label{lemma disj2}
\hfill
\begin{enumerate}
\item
For any $n\geq1$ and $B\subset M$, $\Phi_{f}^n(B)\cap B=\emptyset$ is equivalent to $\Phi_{f}^{-n}(B)\cap B=\emptyset$.
\item
If there is $n\geq1$ and $B\subset M$ such that $\Phi_{f}^{k}(B)\cap B=\emptyset$ for any $1\leq k\leq n$, then
$$
\Phi_{f}^{-r}(B)\cap \Phi_{f}^{-s}(B)=\emptyset,
$$ 
where $1\leq r-s\leq n$.
\end{enumerate}
\end{lemma}

\begin{proof}
From $\Phi_{f}^n(B)\cap B=\emptyset$ we have that, for any $i_1$,
$$
\hat\Phi_{f,i_1}\left(\Phi_{f}^{n-1}(B)\cap \conj M_{i_1}\right)\cap B=\emptyset.
$$
Then its pre-image by $\hat\Phi_{f,i_1}$ is still empty, i.e.
\begin{equation}\label{eq gn-1}
\Phi_{f}^{n-1}(B)\cap \conj M_{i_1} \cap  \hat\Phi_{f,i_1}^{-1}(B)=\emptyset.
\end{equation}
The reverse is also true since~\eqref{eq gn-1} implies that the image of $\Phi_{f}^{n-1}(B)\cap \conj M_{i_1} $ by $\hat\Phi_{f,i_1}$ can not intersect $B$.
Using the fact that $\hat\Phi_{f,i_1}^{-1}(B)\subset \conj M_{i_1}$, we have shown that $\Phi_{f}^n(B)\cap B=\emptyset$ is equivalent to $\Phi_{f}^{n-1}(B)\cap  \hat\Phi_{f,i_1}^{-1}(B)=\emptyset$ for every $i_1$.

Repeating this procedure $n-1$ more times, we get for any $i_1,\dots,i_n$,
$$
B\cap \hat\Phi_{f,i_n}^{-1}\dots \hat\Phi_{f,i_1}^{-1}(B)=\emptyset.
$$
Finally, the union of the above sets is again empty and, by noticing that
$$
\Phi_{f}^{-n}(B)=\bigcup_{i_1,\dots,i_n}\hat\Phi_{f,i_n}^{-1}\dots \hat\Phi_{f,i_1}^{-1}(B),
$$
the proof of the first claim is complete.

By the above, $\Phi_{f}^{-k}(B)\cap B=\emptyset$. Applying $\Phi_{f}^{-s}$ we obtain that $\Phi_{f}^{-(k+s)}(B)\cap \Phi_{f}^{-s}(B)=\emptyset$. Write now $r-s=k$ to get the second claim.
\end{proof}

Below we use the notation $S_n^+(f)$ to highlight the dependence on $f$ for $n\geq2$.
The hypothesis of the next proposition means that there are no trajectories connecting two vertexes (not necessarily distinct) of the polygon formed by $(n-1)$  straight segments.

\begin{proposition}\label{prop branch number cst}
Let $f_0\in\RR_1^1$ or $f_0=0$, and $n\geq 1$.
If 
$$
\Phi_{f_0}^k(S_1^+)\cap S_ 1^+=\emptyset
\te{for}
1\leq k\leq n-1,
$$
then there is $\delta>0$ such that for $f\in\RR_1^1$ satisfying $\lambda(f-f_0)<\delta$ we have $p(S_n^+(f))=p(S_1^+)$.
\end{proposition}

\begin{proof}
Let $f_0\in\RR_1^1$.
By Lemma~\ref{lemma disj2} all pre-images of $S_1^+$ under $\Phi_{f_0}$ up to the iterate $n-1$ are disjoint. 
Moreover, each $\Phi_{f_0,i}^{-1}(S_1^+)\subset M_i$ is diffeomorphic to connected components of $S_1^+$ and the same holds up to the $(n-1)$-th iterate.
Since $f\mapsto \Phi_{f,i}$ and $f\mapsto \Phi_{f,i}^{-1}$ are continuous maps for each $i$, the same properties hold for $f$ with $\lambda(f-f_0)$ small enough.
Recall also the fact that each $\Phi_{f,i}^{-1}$ preserves horizontal lines.
Therefore, the number $p(S_k^+(f))$ is the same for every $1\leq k\leq n$.

Suppose now that $f_0=0$. In this case each $\Phi_{0,i}$ is not a diffeomorphism. For the same reasons as above, given any $f\in\RR^1_1$ such that $\lambda(f)$ is small, $\Phi_{f}^k(S_1^+)\cap S_ 1^+=\emptyset$, $1\leq k\leq n-1$. The proof follows from the previous case.
\end{proof}

\begin{theorem}\label{coroll hyp attr}
Consider a polygon without parallel sides facing each other and $f_0\in\RR_1^1$ or $f_0=0$ such that
\begin{equation*}
\Phi_{f_0}^{k}(S_1^+)\cap S_1^+=\emptyset \te{for} k\geq1.
\end{equation*}
Then there is $\delta>0$ such that for $f\in\BB$ satisfying $\lambda(f-f_0)<\delta$, the billiard map $\Phi_f$ has a hyperbolic attractor with countably many ergodic SRB measures.% and dense hyperbolic periodic points.
\end{theorem}

\begin{proof}
By the uniform hyperbolicity of $\Phi_{f}$ (see Proposition~\ref{pr:delta-zero}), there is some $m\geq1$ such that $\alpha(\Phi^m_{f})>1$. Take the supremum of such $m$ in some neighborhood of $f_0$.
Let $n\geq 1$ such that $\alpha(\Phi^{nm}_{f})\geq \alpha(\Phi^m_{f})^n>p(S_1^+)$.
By Proposition~\ref{prop branch number cst} we have that $p(S_{nm}^+)=p(S_1^+)$ by choosing a smaller $\delta$. 
The claim follows from Theorem~\ref{thm bill p<a}.
\end{proof}

\begin{theorem}\label{thm hyp attr}
Consider a polygon without parallel sides facing each other, $f_0\in\RR_1^1$ or $f_0=0$, 
and
$$
n>\frac{\log p(S_1^+)}{\log\alpha(\Phi_{f_0})}
$$ 
such that
\begin{equation*}
\Phi_{f_0}^k(S_1^+)\cap S_1^+=\emptyset
\te{for} 
1\leq k\leq n-1.
\end{equation*}
Then there is $\delta>0$ such that for $f\in\BB$ satisfying $\lambda(f-f_0)<\delta$, the billiard map $\Phi_f$ has a hyperbolic attractor with countably many ergodic SRB measures.
\end{theorem}

\begin{proof}
Proposition~\ref{prop branch number cst} holds for $\Phi_f$, i.e. $p(S^+_n(f))=p(S^+_1)$, for $\lambda(f-f_0)$ small enough.
From the continuity of $f\mapsto \alpha(\Phi_f)$, for any $\varepsilon>0$ and sufficiently small $\lambda(f-f_0)$, we have
$
\alpha(\Phi_f^n) \geq \alpha(\Phi_f)^n  > (\alpha(\Phi_{f_0})-\varepsilon)^n.
$
By the hypothesis on $n$ we get $\alpha(\Phi_{f_0})>p(S^+_1)^{1/n}$.
So, for $\varepsilon$ small enough (depending on $n$) we have $\alpha(\Phi_f^n)>p(S_1^+)$.
The result now follows from Theorem~\ref{thm bill p<a}.
\end{proof}

The study of the ergodic properties of the SRB measures obtained in Theorems~\ref{coroll hyp attr} and \ref{thm hyp attr} will appear elsewhere~\cite{MDDGP13}.

\section{Generic polygons}
\label{sec:Generic}
%!TEX root = dissipative_main.tex

%%%%%%%%%%%
We now introduce the moduli spaces of polygons and prove that, generically, 
polygons in these spaces have no ``orthogonal vertex connections''. This will be used later to prove the existence of SRB measures for small $\lambda(f)$.
 
\subsection{Moduli spaces of polygons}

Any sequence $P=(v_0,v_1,\ldots, v_{n-1})$ of points in $\Rr^2$
determines a {\em closed polygonal line}. 
The line segments $e_1=[v_{0},v_1]$,  $e_2=[v_{1},v_2]$, $\ldots$, $e_0=e_n=[v_{n-1},v_0]$,
are called the {\em edges} of $P$.
We say that $P$ is a {\em $n$-gon} if
the lines supporting $e_{i-1}$ and $e_i$ are always distinct, and the polygonal
line $P$ is non self-intersecting.
Two $n$-gons $P=(v_0,v_1,\ldots, v_{n-1})$,
$P'=(v_0',v_1',\ldots, v_{n-1}')$ are said to be {\em similar} if 
there is an orientation preserving similarity 
$T$ of the Euclidean plane such that $T(v_i)=v_i'$, for $i=0,1,\ldots, n-1$. 
Similarity is an equivalence relation on the space of all $n$-gons,
and the quotient   by this relation is called
the {\em moduli space}  of $n$-gons, denoted hereafter by $\mathscr{P}_n$.
Let $\Pp^1$ and $\Pp^2$ denote the real projective line and plane,
respectively.

\begin{proposition}\label{moduli:Pscr:n}
The moduli space $\mathscr{P}_n$ is diffeomorphic to an open semialgebraic subset of 
$\Pp^1\times (\Pp^2)^{n-3}\times \Pp^1$, and it is a manifold
of dimension $2n-4$.
\end{proposition} 

\begin{proof} 
Each edge $e_i$ determines a line in the Euclidean plane $\Rr^ 2$,
defined by an affine equation $a_i x+ b_i y+c_i=0$, and we shall refer to 
$(a_i:b_i:c_i)$ as the projective coordinates of $e_i$.
The vertices $v_0, v_1,\ldots, v_{n-1}$ of the polygon,
$v_i = e_{i-1}\cap e_i$,  are easily computed to have coordinates
\begin{equation} \label{xi:yi}
 (x_i,y_i) =\left(\frac{b_{i-1} c_i-b_i c_{i-1}}{a_{i-1} b_i-a_i b_{i-1}},\frac{a_i c_{i-1}-a_{i-1} c_i}{a_{i-1} b_i-a_i b_{i-1}}\right)\;.
\end{equation}
We are adopting the convention that $e_{0}=e_{n}$.
Up to a similarity transformation, we can assume that
$v_{0}=(-1,0)$ and $v_1=(1,0)$, and specify coordinates for the edges $e_0$, $e_1$ and $e_{n-1}$
respectively of the form\,
$(a_0:b_0:c_0)=(0,1,0)$, $(a_1:b_1:a_1)\equiv(a_1:b_1)$ and
$(a_{n-1},b_{n-1},-a_{n-1})\equiv (a_{n-1}:b_{n-1})$. 
Hence any $n$-gon $P$ determines and is determined by the 
vector in $\Pp^1\times (\Pp^2)^{n-3}\times \Pp^1$,
$$\theta(P)=[(a_1:b_1), (a_2:b_2:c_2), \ldots, (a_{n-2}:b_{n-2}:c_{n-2}), (a_{n-1}:b_{n-1})]\;.  $$

The values of the projective coordinates of the edges missing  in the vector $\theta(P)$ are fixed to 
be $(a_0:b_0:c_0)=(0:1:0)$, $c_1=a_1$ and  $c_{n-1}=-a_{n-1}$.
Since consecutive edges are not parallel, we must have
\begin{equation} \label{non:parallel}
a_{i-1} b_i-a_i b_{i-1}\neq 0\quad \text{ for } \; i=0,1,\ldots, n-1 \;.
\end{equation}
Under these conditions, the coordinates $(x_i,y_i)$ of the vertices in~ (\ref{xi:yi})
are well defined.
The second condition in the definition of a $n$-gon is equivalent to saying that for $i\neq j$ and $i\neq j\pm 1$
either the lines $(a_i:b_i:c_i)$ and $(a_j:b_j:c_j)$ are parallel, or else
their intersection point
$$ (u_{ij},v_{ij}) =\left(\frac{b_{j} c_i-b_i c_{j}}{a_{j} b_i-a_i b_{j}},\frac{a_i c_{j}-a_{j} c_i}{a_{j} b_i-a_i b_{j}}\right)\; $$
does not lie in $e_i\cup e_j$, a condition which is easily seen to be semialgebraic.
Hence the subset $\Omega_n\subseteq \Pp^1\times (\Pp^2)^{n-3}\times \Pp^1$ of all data
satisfying both~(\ref{non:parallel}) and the previous condition is an open semialgebraic subset
of $\Pp^1\times (\Pp^2)^{n-3}\times \Pp^1$.
The previous considerations show that the map $\theta$ induces a diffeomorphism
$\theta\colon\mathscr{P}_n\to\Omega_n$.
\end{proof}

Given a $n$-gon $P$ with edges $e_1,\ldots, e_n$,  denote by $\ell_i$ the line
supporting $e_i$, and by  $\pi_i:\Rr^2\to \ell_i$
the orthogonal projection onto $\ell_i$.
We say that $P$ has an {\em orthogonal vertex connection of order $m$} if there is a sequence of edges  
$e_{i_1},\ldots, e_{i_{m'}}$,  
with length $m'\leq m$, and a sequence $p_0,\ldots, p_{m'}$ of points in $P$ such that:
\begin{enumerate}
\item[(a)] $p_0$ and $p_{m'}$ are vertices of $P$;
\item[(b)] $p_k= \pi_{i_k}(p_{k-1})$, for every $k=1,\ldots, m'$;
\item[(c)] the polygonal line
$p_0 p_1\ldots  p_{m'}$ is contained in $P$;
\item[(d)] the points
$p_0,\ldots, p_{m'}$ are not all equal.
\end{enumerate}
When this property holds for some $m\in\Nn$, we simply say that  $P$ has an {\em orthogonal vertex connection}.
We denote by $\mathscr{S}_{n,m}^\ast\subseteq \mathscr{P}_{n}$ the subset of  $n$-gons with 
 some orthogonal vertex connection of order $m$;
and set $\mathscr{S}_{n,\infty}^\ast$ to be the union of all $\mathscr{S}_{n,m}^\ast$. 

\begin{figure*}
\begin{center}
\includegraphics*[scale=.4]{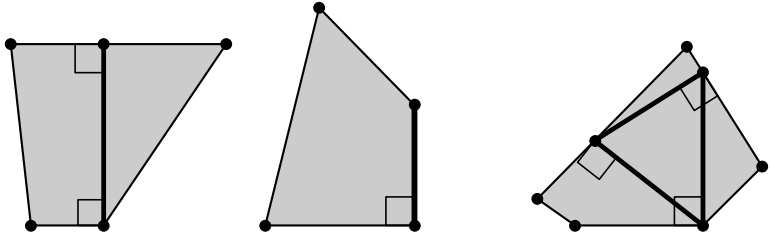}
\end{center}
\caption{Examples of polygons with orthogonal vertex
connections.}
\end{figure*}

\begin{remark}
A polygon $P$ is in $\mathscr{S}_{n,\infty}^\ast$ iff 
some vertex of $P$ has a forward orbit under the slap map
which eventually hits another vertex. Here we exclude fixed points at vertices where the polygon makes acute angles,
because of condition (d).
\end{remark}

\begin{proposition} The set $\mathscr{S}_{n,\infty}^\ast$ (resp. $\mathscr{S}_{n,m}^\ast$) is a countable (resp. finite) union of codimension $1$, closed semialgebraic sets.
\end{proposition}

\begin{proof}
The orthogonal projection onto an edge $e$ with projective coordinates
$(a:b:c)$ is given by
\begin{equation} \label{orth:proj}
\pi_{(a:b:c)}(x,y)=\left(x-  \frac{ a x + b y + c }{a^ 2+b^ 2}\, a, \, 
x- \frac{ a x + b y + c }{a^ 2+b^ 2}\, b \,\right) \;.
\end{equation}
Identify  $\mathscr{P}_n$ with the open subset $\Omega_n\subseteq \Pp^1\times (\Pp^2)^{n-3}\times \Pp^1$, as explained in proposition~ \ref{moduli:Pscr:n}.
The set $\mathscr{S}_{n,m}^\ast$
is contained in the finite union of algebraic 
hypersurfaces defined by 
\begin{equation} \label{ovcp:eq}
(x_\ell,y_\ell)= \pi_{(a_{i_m'}:b_{i_m'}:c_{i_m})} \circ \ldots \circ
\pi_{(a_{i_1}:b_{i_1}:c_{i_1})}  (x_k,y_k) 
\end{equation}
where $1\leq m'\leq m$, $(i_0=k, i_1,\ldots, i_{m'})$ is a  finite sequence of indices in $\{0,1,\ldots,n-1\}$,
and $v_\ell=(x_\ell,y_\ell)$ is an endpoint of $e_{i_m}$.
This equation expresses the existence of an orthogonal vertex connection of order $m$
between the vertices $v_k$ and $v_\ell$.
Because expressions in~(\ref{xi:yi}) and~(\ref{orth:proj}) are rational functions
of the coordinates $(a_i:b_i:c_i)$ of the polygon's edges, equation 
 ~(\ref{ovcp:eq}) is also rational in these coordinates. Re\-du\-cing to a common denominator, and eliminating it, equation~(\ref{ovcp:eq}) becomes polynomial
 %, of the form $Q_{k,i_1,\ldots, i_{m'},\ell}=0$,
in the projective coordinates of $\Omega_n\subseteq \Pp^1\times(\Pp^{n-3})\times\Pp^1$.
Notice that ~(\ref{ovcp:eq}) is a system of two equations
which reduces to single equation because the projection points  lie in a line.
Given distinct vertex indices $i,j$,
and an edge index $k$, consider the set $\Sigma_{k,i_1,\ldots, i_{m'},\ell}$ of polygons with an orthogonal vertex connection between the vertices
$v_k$ and $v_\ell$, having the prescribed itinerary.
The set  $\Sigma_{k,i_1,\ldots, i_{m'},\ell}$  is
defined by the equation~ (\ref{ovcp:eq}) and hence is 
a codimension one  closed 
subset of the semialgebraic set $\mathscr{P}_n$.
Thus $\mathscr{S}_{n,m}^\ast$ is a finite union,
and $\mathscr{S}_{n,\infty}^\ast=\cup_{m\geq 1} \mathscr{S}_{n,m}^\ast$ a countable union, of such closed semialgebraic sets. 
\end{proof}

\begin{corollary}\label{corollary generic novcp}
The set  $\mathscr{P}_{n}\setminus  \mathscr{S}_{n,m}^\ast$
is open and dense, whereas the set  $\mathscr{P}_{n}\setminus \mathscr{S}_{n,\infty}^\ast$  is residual in $\mathscr{P}_{n}$. 
Both  sets
$\mathscr{S}_{n,m}^\ast$ and $\mathscr{S}_{n,\infty}^\ast$ have zero Lebesgue measure in $\mathscr{P}_{n}$.
\end{corollary}
%%%%%%%%%%%%%%%%%%%%%%%%%%%%%%%%%%%%%%%%%%%%%

\subsection{Strongly contracting generic polygonal billiards}

The properties of the degenerate case $f=0$, simpler to obtain, can be extended by Theorem~\ref{thm hyp attr} to billiards with $\lambda(f)$ very close to $0$.

One can naturally reduce $\Phi_0$ on $\{\theta=0\}$ to a one-dimensional map called the slap map.
This map is piecewise affine because the billiard is polygonal.
Wherever a vertex of the polygon projects orthogonally to the interior of a side we obtain an element of the set $S$ defined by $S_1^+\cap\{(s,\theta)\colon \theta=0\}=S\times\{0\}$. 
These are the pre-images of points where the slap map is not defined.
On each connected open interval $I_i$ of the domain of $\Phi_0$ we define $\Phi_{0,i}=\Phi_0|_{I_i}$.

For any polygon we define the angles $\varphi_i\in(-\pi/2,\pi/2)$ between the straight lines supporting the sides of the polygon.
If there are parallel sides we set the corresponding angle to $\varphi_i=0$.
It is simple to check that the minimum expansion rate of $\Phi_0$ is given by
$$
\alpha(\Phi_0)=\min_i\left|\cos\varphi_i \right|^{-1} \geq  1.
$$
Note that $\alpha(\Phi_0)>1$ \, iff\, there are no parallel sides.

\begin{theorem}
For any polygon $P\in\mathscr{P}_{n}\setminus \mathscr{S}_{n,\infty}^\ast$, there is $\lambda_0>0$ such that if $\lambda(f)<\lambda_0$ and $f\in\BB$, then $\Phi_f$ has a generalized hyperbolic attractor with finitely many ergodic SRB measures and dense hyperbolic periodic points.
\end{theorem}

\begin{proof}
Notice that $\Phi_0^k(S)\cap S\neq \emptyset$ 
implies the existence of an  orthogonal vertex connection of order $k$.
Hence, if a polygon does not have any orthogonal vertex connection,
 then the corresponding slap map satisfies $\Phi_0^k(S)\cap S=\emptyset$ for every 
$k\geq 1$. Next observe that the number $p(S_1^+)$ depends only on the polygon $P$,
and not on the reflection law $f$. Thus taking $n\gg  {\log p(S_1^ +)}/{\log\alpha(\Phi_0)}$,
for small enough $\lambda_0>0$ and every $f\in\BB$ with $\lambda(f)<\lambda_0$ one has
\, $n > {\log p(S_1^+)}/{\log\alpha(\Phi_f)}$, and \, 
$\Phi_f^k(S_1^+)\cap S_1^+=\emptyset$, for every $1\leq k\leq n-1$.
The assumptions of Theorem~\ref{thm hyp attr} are met, and the theorem follows.
\end{proof}

\section{Regular polygons}
\label{sec:regpoly}
%!TEX root = dissipative_main.tex

%%%%%%%%%%%%%%%%%%%%%%%%%%%

Throughout this section $P$ is a regular $d$-gon. We assume that the sides of the polygon are normalized to a unit size. The phase space of the billiard map is thus $$M=(0,N)\times\left(-\frac\pi 2,\frac\pi 2\right)\,.$$

%%%%%%%%%%%%%%%%%%%%
\subsection{Phase space reduction}

The study of the dynamics in regular polygonal billiards can be simplified using the group of symmetries of the polygon. In the phase space this reduction can be achieved by simply identifying all sides of the polygon. So we define an equivalence relation $\sim$ on $M$ by $(s_1,\theta_1)\sim(s_2,\theta_2)$ if and only if $s_1-s_2\in\mathbb{Z}$. Let $$\widetilde{M}=(0,1)\times(-\pi/2,\pi/2)\,.$$ Clearly, the quotient space $M/\sim$ can be identified with $\widetilde{M}$, which we call the \textit{reduced phase space}. The induced billiard map $\phi_f$ on $\widetilde{M}$ is the \textit{reduced billiard map}. Notice that $M$ is a $d$-fold covering of $\widetilde{M}$ and that the reduced billiard map $\phi_f$ is a factor of $\Phi_f$.

%%%%%%%%%%%%%%%%%%%%
\subsection{The reduced billiard map}

A straightforward computation reveals that the singular set of the reduced billiard map is 
$$
\widetilde{N}^{+}_{1} = \partial \widetilde{M} \cup \widetilde{S}^{+}_{1}\,,
$$
where $\widetilde{S}_1^+$ is the union of the graph of the curves
$$
\gamma_k(\theta)=\frac{\sin(\theta_{k-1}-\theta)\cos(\theta_k)}{\sin(\theta_{k-1}-\theta_{k})\cos(\theta)} \,,\quad k=2,\ldots,d-1\,,
$$
and  $\theta_k=\pi/2-k\pi/d$ which defines a partition of the interval $(-\pi/2,\pi/2)$ into $d$ subintervals. 
%It is clear that $\gamma_k(\theta_{k-1})=0$ and $\gamma_k(\theta_{k})=1$.% 
Let, $$ A_k=\left\{(s,\theta)\in \widetilde{M}\setminus\widetilde{N}_1^+\,:\, \gamma_{k+1}(\theta)<s<\gamma_{k}(\theta)\right\}\,,\quad k=1,\ldots, d-1\,.$$
Note that $\widetilde{M}\setminus \widetilde{N}^+_1 = \cup_{k=1}^{d-1} A_k$. 
%(see Fig.~\ref{figures_redphasespace})
 We then obtain the \textit{$k$-branch} of the reduced billiard map
\begin{equation}\label{eq:redmap}
\left.\phi_f\right|_{A_k}(s,\theta)=\left(\frac{\gamma_k(\theta)-s}{\gamma_k(\theta)-\gamma_{k+1}(\theta)},f(2\theta_k -\theta)\right)
\end{equation}
where $f$ is a contracting reflection law.

\begin{remark}\label{rem:oddf}
If the contracting reflection law $f$ is an odd function then $\phi_f$ commutes with the involution $T(s,\theta)=(1-s,-\theta)$, i.e. $$T\circ\phi_f=\phi_f\circ T\,.$$ Also note that $T(A_k)=A_{d-k}$.
\end{remark}

When $f=0$ we obtain the one-dimensional \textit{reduced slap map} $\phi_0$. If $P$ is an even $d$-gon then $\phi_0(s)=1-s$. Otherwise,
$$
\phi_0(s)= -\frac{1}{\beta}\,\left(s-\frac{1}{2}\right) +\varepsilon(s), 
$$
where $\beta = \cos(\frac{\pi}{d})$ and $\varepsilon\colon[0,1)\to \{0,1\}$ is the function
$$
\varepsilon(s)=
\begin{cases} 
1, &  s\geq \frac{1}{2} \\
0, & s < \frac{1}{2}.
\end{cases}
$$

%%%%%%%%%%%%%%%%%%%%%%%%%%%%%%%%%%%%%%%%%%%%%
\subsection{Regular polygons with an odd number of sides}
 
\begin{proposition}\label{slap map not pre-per}
For any odd number $d\geq 5$ the point $s=1/2$ is not a pre-periodic point of the reduced slap map $\phi_0$.
\end{proposition}

\begin{proof}
Let us write $\varepsilon_j(s)=\varepsilon(\phi_0^j(s))$ and $\delta_j(s)=2\,\varepsilon_j(s)-1$. 
Notice that we always have $\delta_j(s)\in\{-1,1\}$.
Then, by induction we get
$$ 
\phi_0^ n(s)=  \frac{1}{2} +\frac{(-1)^n}{\beta^ n}\,\left(s-\frac{1}{2}\right) + 
\frac{1}{2\,\beta^ {n-1}}\, \sum_{j=0}^{n-1} 
(-1)^{n-j-1}\,\delta_{j}(s)\,\beta^{j}.
$$
Denote by $P_n(s,\beta)$ the tail term above. 
More precisely, set
$$ 
P_n(s,\beta)= \sum_{j=0}^{n-1} 
(-1)^{n-j-1}\,\delta_{j}(s)\,\beta^{j}.
$$

Assume, by contradiction, that $\phi_0^n(1/2)=\phi_0^ k(1/2)$ for some $n>k\geq 0$. 
Then,
$$
\frac{1}{\beta^ {n-1}}\,P_n\left(1/2,\beta\right)=\frac{1}{\beta^ {k-1}}\,P_k\left(1/2,\beta\right),
$$
which is equivalent to
$$
P_n\left(1/2,\beta\right)-\beta^{n-k} P_k\left(1/2,\beta\right)=0.
$$
The last expression, hereafter denoted by $Q(\beta)$, is a polynomial in $\beta$ of degree $n-1$ with all coefficients in $\{-2,-1,0,1,2\}$.
Moreover, since $n>k$, $Q(0)=P_n(\frac{1}{2},0)=\pm 1$.
Because $\tilde{\Psi}_{2d}(x)$ is the minimal polynomial of $\beta$ (see the Appendix)
there is a factorization $Q(x)=\tilde{\Psi}_{2d}(x)\, S(x)$ in $\Q[x]$.
Since $Q(x)\in\Z[x]$, Gauss lemma implies there is some $c\in\Q\setminus\{0\}$ such that
$c\,\tilde{\Psi}_{2d}(x)\in\Z[x]$ and $c^{-1}\, S(x)\in\Z[x]$.
Since, by Corollary~ \ref{Psin:coef}, $\tilde{\Psi}_{2d}(x)$ has a unit constant term $\tilde{\Psi}_{2d}(0)=\pm 1$, 
it follows that $c\in\Z$, and hence  $S(x)= c\,(c^{-1} S(x))\in\Z[x]$.
Thus, the leading coefficient of $Q(x)$ must be a multiple of
the leading coefficient $\tilde{\Psi}_{2d}(x)$.
But the first is $\pm 1$ or $\pm 2$, while by Corollary~ \ref{Psin:coef}
the second has leading coefficient at least $4$, which is obviously impossible.
\end{proof}

\begin{corollary}
Every odd sided regular polygon has the no orthogonal vertex connection property (see Section~\ref{sec:Generic} for the definition). 
\end{corollary}

\begin{proof}
Any equilateral triangle has this property because
its vertexes correspond to fixed points of the slap map.
Moreover, for any odd number $d\geq 5$, if the $d$-sided regular
polygon had a `orthogonal vertex connection' this would
imply that $s=1/2$ is a pre-periodic point of the reduced slap map $\phi_0$, thus contradicting proposition~\ref{slap map not pre-per}.
\end{proof}

\begin{theorem}
Consider a regular polygon with an odd number of sides.
There is $\lambda_0>0$ such that if $\lambda(f)<\lambda_0$ and $f\in\BB$, then the billiard map $\Phi_f$ has a generalized hyperbolic attractor with finitely many ergodic SRB measures and dense hyperbolic periodic points.
\end{theorem}

\begin{proof}
The set of singularities of the slap map includes the middle points on each of the sides of the polygon. The claim is a consequence of Theorem~\ref{thm hyp attr} and Proposition~\ref{slap map not pre-per}.
\end{proof}

Notice that for the equilateral triangle Theorem~\ref{thm acute triangles} is an improvement of the above result.

%%%%%%%%%%%%%%%%%%%%%%
\subsection{Regular polygons with an even number of sides}

The approach used to prove the existence of SRB measures for odd sided regular polygons does not work with even sided regular polygons, because these polygons have parallel sides and so the slap map is not expanding. However, this fact does not prevent the existence of hyperbolic attractors as proved for the square in Section~\ref{sec:rectangles} and shown by numerical experiments in \cite{MDDGP12}.

In this section we show that, under certain conditions on the contracting reflection law, every orbit whose angle of incidence is sufficiently small will eventually be trapped between two parallel sides. 

Consider a regular polygon with an even number $d\geq 6$ of sides. The square is considered in Section~\ref{sec:rectangles}. 

\begin{proposition}\label{prop:evenN}
Let $f\in\RR^k_1$ be an odd function such that $f'>0$, $\lambda(f)\leq1/2$ and $f(\delta\frac{2\pi}{d})\leq\delta f(\frac{2\pi}{d})$ for every $0\leq\delta\leq1$. Then there exists a positive constant $C$ such that every orbit of $\Phi_f$ having more than $C$ collisions between parallel sides belongs to the basin of attraction of $\PP$.
\end{proposition}

\begin{proof}
In the reduced phase space the attractor $ \mathcal{P} $ is the interval $(0,1)$. Moreover, since the contracting reflection law is an odd function, the set
$$
B=\left\{(s,\theta)\in \widetilde{M}\,:\, \sigma_\infty(\theta)-1<s<\sigma_\infty(\theta)\right\}\,,
$$
where 
$$
\sigma_\infty(\theta)=1-\cot\left(\frac{\pi}{d}\right)\sum_{n=0}^\infty\tan(f^n(\theta))\,,
$$
is a trapping region, i.e., $\phi_f(B)\subset\mathrm{int}(B)$. In fact, positive semi-orbits starting in $B$ correspond to billiard trajectories which are trapped between any two parallel sides. 

\begin{figure}
  \begin{center}
    \includegraphics[width=3in]{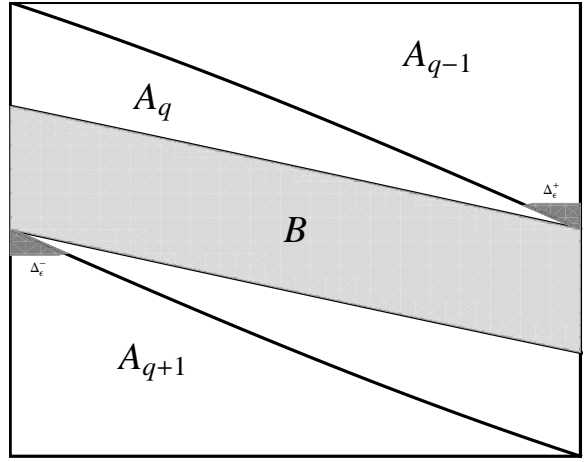}
  \end{center}
  \caption{The set $\Delta_\epsilon$.}
  \label{figure_delta}
\end{figure}

Given $\epsilon>0$, let $H_\epsilon$ denote the horizontal strip of width $\epsilon$ around $\theta=0$, i.e. $H_\epsilon=\left\{(s,\theta)\in\widetilde{M}:\left|\theta\right|<\epsilon\right\}$. Also let
$$
\Delta_\epsilon^{\pm}=A_{q\mp1}\cap H_\epsilon\,\quad\text{and}\quad\Delta_\epsilon=\Delta_\epsilon^-\cup\Delta_\epsilon^+\,,
$$
where $d=2 q$. These sets are depicted in Fig.~\ref{figure_delta}. 
We claim that there exists an $\epsilon>0$ such that \begin{equation}\label{eq:phideltaclaim}\phi_f^2(\Delta_\epsilon)\subset B\,.
\end{equation}
Suppose that the claim holds true. It is clear that for $\epsilon$ sufficiently small $\Delta_\epsilon\subset\phi_f(A_q\setminus\mathrm{cl}(B))$. Thus, take
$$
C=\min\left\{n\geq1\colon \phi_f^{-n}(\Delta_\epsilon)\nsubseteq A_q\right\}\,.
$$

If the positive semi-orbit of $x\in \widetilde{M}$ has more than $C$ collisions between parallel sides then, by the definition of $C$, there exists an $n\geq1$ such that $\phi_f^n(x)\in\Delta_\epsilon$. It follows from the claim that $\phi_f^{n+2}(x)\in B$. Thus proving the proposition. 

Therefore, it remains to prove \eqref{eq:phideltaclaim}. Since $f$ is odd, we have by Remark~\ref{rem:oddf} that $\phi_f\circ T= T\circ \phi_f$. Thus, it is sufficient to prove that there exists an $\epsilon>0$ sufficiently small such that $\phi_f^2(\Delta_\epsilon^+)\subset B$.

For $q\geq3$ the map $\left.\phi_f\right|_{A_{q-1}}$ extends continuously to the boundary of $A_{q-1}$. Indeed, by \eqref{eq:redmap} we have
\begin{equation}\label{eq:phifcontext}
\left.\phi_f\right|_{A_{q-1}}(s,\theta)=\left(\cot\left(\frac{\pi}{2q}\right)\tan\left(g(\theta)\right)-s\cos(\theta)\sec(g(\theta)),f(g(\theta))\right)\,,
\end{equation}
where $g(\theta)=\frac{\pi}{q}-\theta$. To simplify the notation, let us denote the unique extension of $\left.\phi_f\right|_{A_{q-1}}$ by $\phi_{f,q-1}$. 

In order to prove the inclusion $\phi_f^2(\Delta_\epsilon^+)\subset B$ for $\epsilon$ sufficiently small, we only need to trace two iterates of the point $(s,\theta)=(1,0)$ under the continuous extension $\phi_{f,q-1}$. In fact, if $\phi_{f,q-1}(1,0)\in \mathrm{cl}(A_{q-1})$ and $\phi_{f,q-1}^2(1,0)\in B$ then by continuity of $\phi_{f,q-1}$, the hypothesis $f'>0$ and the fact that $\text{dist}(\Delta_\epsilon^+,(1,0))\rightarrow 0$ as $\epsilon\rightarrow 0$ we obtain the desired inclusion.  

In the following we compute two iterates of the point $(s,\theta)=(1,0)$. Taking into account \eqref{eq:phifcontext}, we have $$\phi_{f,q-1}(1,0)=\left(1,f(\pi/q)\right)\,.$$ Since $\lambda(f)\leq 1/2$ we have that $\phi_{f,q-1}(1,0) \in \mathrm{cl}(A_{q-1})$. Now we show that $\phi_{f,q-1}^2(1,0)\in B$. A simple computation reveals that 
\begin{equation*}
\phi_{f,q-1}^2(1,0)=\left(\tan(\hat{\theta})\left(\cot\left(\frac{\pi}{2q}\right)-\sin\left(\frac{\pi}{q}\right)\right)-\cos\left(\frac{\pi}{q}\right),f(\hat{\theta})\right) 
\end{equation*}
where $\hat{\theta}=g(f(\frac{\pi}{q}))$. Taking into account the definition of the set $B$, the inclusion $\phi_{f,q-1}^2(1,0)\in B$ holds if and only if
\begin{equation}\label{ineqB}
\cot\left(\frac{\pi}{2q}\right)\sum_{n=0}^\infty\tan(f^{n+1}(\hat{\theta}))<1+\tan(\hat{\theta})\sin\left(\frac{\pi}{q}\right)+\cos\left(\frac{\pi}{q}\right)\,.
\end{equation}
Let $\lambda_q=f(\pi/q)/(\pi/q)$. Note that $0\leq\lambda_q\leq 1/2$. Since $f(\delta\frac{\pi}{q})\leq \delta f(\frac{\pi}{q})$ for every $0\leq\delta\leq 1$ we obtain the following estimate
$$
f^{n+1}(\hat{\theta})=f^{n+1}\left(\frac{\pi}{q}(1-\lambda_q)\right)\leq \frac{\pi}{q}\lambda_q^{n+1}(1-\lambda_q)\,,\quad n=0,1,\ldots
$$
Thus,
$$
\tan(f^{n+1}(\hat{\theta}))\leq2\tan\left(\frac{\pi}{2q}\right)\lambda_q^{n+1}(1-\lambda_q)\,,
$$
by convexity of the tangent function. Using the previous estimate we get,
$$
\cot\left(\frac{\pi}{2q}\right)\sum_{n=0}^\infty\tan(f^{n+1}(\hat{\theta}))\leq 2\lambda_q\leq1\,.
$$
Since $\tan(\hat{\theta})\sin(\pi/q)+\cos(\pi/q)>0$, inequality \eqref{ineqB} holds. 
\end{proof}

\begin{corollary}
Under the assumptions of Proposition~\ref{prop:evenN}, every invariant set $ \Sigma $ of $ \Phi_{f} $ not intersecting the basin of attraction of $ \PP $ is uniformly hyperbolic. \end{corollary}

\begin{proof}
Since $ \Sigma $ does not intersect the basin of attraction of $\PP$ then by Proposition~\ref{prop:evenN} there exists a constant $C>0$ such that the positive semi-orbit of every point $x\in \Sigma$ has no more than $C$ collisions between parallel sides. Thus we can apply Proposition~\ref{pr:delta-zero} and obtain the desired result. 
\end{proof}

\begin{remark}
The assumptions of Proposition~\ref{prop:evenN} are satisfied for the linear contracting reflection law $f(\theta)=\sigma\theta$ where $\sigma\leq1/2$.
\end{remark}
 
\begin{remark}
In Proposition~\ref{prop:evenN}, the contracting reflection law is assumed to be an odd function. This assumption is not necessary, it only simplifies certain estimates in the proof. In fact, we can remove it provided $f(-\delta \frac{2\pi}{d})\geq \delta f(-\frac{2\pi}{d})$ for every $0\leq\delta\leq 1$. 
\end{remark}

\section{Acute triangles}
\label{sec:triangles}
%!TEX root = dissipative_main.tex

%%%%%%%%%%%%%%%%%%%%%%%%%%%%%%%%%%%%%%%%%%%%%

We now deal with the existence of SRB measures for contracting billiards in acute triangles (all the internal angles $\varphi_i$ are less than $\pi/2$). The existence and uniqueness of an SRB measure for the case of the equilateral triangle and $f(\theta)=\sigma\theta$ with $\sigma<1/3$ was first obtained in~\cite{arroyo12} by constructing a proper Markov partition. 

\begin{theorem}\label{thm acute triangles}
For any acute triangle and $f\in\BB$ satisfying
$$
\lambda(f) < \frac2\pi \min_i\left(\frac\pi2-\varphi_i\right),
$$ 
the billiard map $\Phi_f$ has a generalized hyperbolic attractor with countably many ergodic SRB measures.
\end{theorem}

\begin{proof}
Let $\tilde D_f $ be the invariant set 
in Lemma~\ref{le:bounded-angle}, and let $V'=V\cap \tilde D_f$. Hence $\Phi_f(S_1^+)\subset V'$. Since $S_1^+$ intersects $V$ at $ \theta = \pi/2-\varphi_i$, it is not difficult to see that if $\lambda(f)< 2 \min_i\left(\pi/2-\varphi_i\right)/\pi $, then $S_1^+\cap V'=\emptyset$ and $\Phi_f(V')\subset V'$.
Therefore, $\Phi_f^k(S_1^+)\cap S_1^+=\emptyset$ for every $k\geq1$. 
The wanted conclusion now follows from Theorem~\ref{coroll hyp attr}.
\end{proof}

%%%%%%%%%%%%%%%%%%%%%%%%%%%%%%%%%%%%%%%%%%%%%

\section{Rectangles}
\label{sec:rectangles}
%!TEX root = dissipative_main.tex

Throughout this section, the polygon $ P $ is a rectangle. Recall that $ \mathcal{P} $ denotes the set of periodic points of period two of $ \Phi $, and let $ B(\mathcal{P}) $ be the basin of attraction of $ \mathcal{P} $. In the following, we show that although $ \mathcal{P} \neq \emptyset $, the invariant set $ D \setminus B(\mathcal{P}) $ may still contain non-trivial hyperbolic sets. We also give sufficient conditions for $ B(\mathcal{P}) = M \setminus N^{+}_{\infty} $.

We start by showing that $ D \setminus B(\mathcal{P}) $ is always hyperbolic.

\begin{proposition}
	\label{pr:rectangle}
	The set $ D \setminus B(\mathcal{P}) $ is hyperbolic.
\end{proposition}

\begin{proof}
A sequence of consecutive collisions $ \{(s_{i},\theta_{i})\}^{m}_{i=0} $ with $ m \ge 1 $ is called a \emph{block} if its first $ m $ collisions occur at two parallel sides of $ P $, whereas its last collision occurs at a side adjacent to the other two; when $ m=1 $, a block consists of two consecutive collisions at adjacent sides of the rectangle. 
It is easy to see that $ |\theta_{m-1}| \le \lambda^{m-1} |\theta_{0}| $, and using \eqref{eq:angle-after}, that 
	\begin{equation}
	\label{eq:collision-angles}
|\bar{\theta}_{m} - \theta_{m-1} | = \frac{\pi}{2}.
	\end{equation}

Since $ \rho \ge 1 $, it follows that $ \Lambda_{m}(s_{1},\theta_{1}) \ge \rho(\bar{\theta}_{m}) $. Next, we want to estimate $ \rho(\bar{\theta}_{m}) $ from below. Let 
	\[
	m_{\lambda} = \frac{2}{\log \lambda} \log \left( \frac{2}{\pi} \cos \frac{\pi \lambda}{2}\right) \qquad \text{and} \qquad b_{\lambda} = r\left(\frac{\pi}{2} (1-\lambda)\right).
	\]
We consider separately the cases: $ m \le m_{\lambda} $ and $ m > m_{\lambda} $. First, suppose that $ m \le m_{\lambda} $. By Lemma \ref{le:bounded-angle}, we know that $ |\theta_{m-1}| < \pi \lambda/2 $. Using \eqref{eq:collision-angles}, we obtain $ |\bar{\theta}_{m}| > \pi (1-\lambda)/2 $ so that
		\begin{equation}
			\label{eq:N-greater-m}
		\Lambda_{m}(s_{1},\theta_{1}) \ge b_{\lambda} \ge b_{\lambda}^{\frac{m}{m_{\lambda}}}.
		\end{equation}
Now, suppose that $ m > m_{\lambda} $. Using the bounds $ |\theta_{0}| < \lambda \pi/2 $ and $ |\theta_{m}| < \lambda \pi/2 $ provided by Lemma \ref{le:bounded-angle}, we get
		\begin{equation}
		\label{eq:cos}
		 \cos (|\bar{\theta}_{m}|) \le 
		\sin (\lambda^{m-1} |\theta_{0}|) \\
		< \sin \frac{\pi \lambda^{m}}{2} \le \frac{\pi \lambda^{m}}{2},
		\end{equation}
		and
		\begin{equation}
		\label{eq:cos-bar}
		\cos (|\theta_{m}|) \ge \cos \frac{\pi \lambda}{2}.
		\end{equation}
		Since $ m > m_{\lambda} $, it follows from the definition of $ m_{\lambda} $ that 
		\[
		\pi \lambda^{m/2} \le 2 \cos \frac{\pi \lambda}{2}.
		\] 
		This fact, together with \eqref{eq:cos-bar} and \eqref{eq:cos}, implies that 
		\begin{equation}
			\begin{split}
			\label{eq:N-smaller-m}
		\Lambda_{m}(s_{1},\theta_{1}) & \ge \rho(\bar{\theta}_{m}) > \frac{2}{\pi \lambda^{m}} \cos \frac{\pi \lambda}{2} \\ & \ge \lambda^{-\frac{m}{2}} \frac{2}{\pi \lambda^{\frac{m}{2}}} \cos \frac{\pi \lambda}{2} \ge \lambda^{-\frac{m}{2}}.
			\end{split}
		\end{equation}
Let $ \mu = \min \{b_{\lambda}^{1/m_{\lambda}},\lambda^{-1/2}\} > 1 $. Combining  \eqref{eq:N-greater-m} and \eqref{eq:N-smaller-m} together, we conclude that
		\begin{equation}
			\label{eq:lambda-m}
		\Lambda_{m}(s_{1},\theta_{1}) \ge \mu^{m}.
		\end{equation}

Suppose that $ (s_{0},\theta_{0}) \in D \setminus B(\mathcal{P}) $. Then for $ n > 0 $ sufficiently large, the sequence of collisions $ \{(s_{i},\theta_{i})\}^{n}_{i = 0} $ can be decomposed into $ k(n) \ge 1 $ consecutive blocks followed by $ j(n) \ge 0 $ consecutive collisions with two parallel sides of the rectangle. Let $ L(s_{0},\theta_{0};n) $ be positive integer such that $ x_{L(s_{0},\theta_{0};n)} $ is the last collision of the $ k(n) $th block. Hence 
\[
%	\label{eq:last}
L(s_{0},\theta_{0};L(s_{0},\theta_{0};n)) = L(s_{0},\theta_{0};n) \le n.
\]
Since $ (s_{0},\theta_{0}) \notin B(\mathcal{P}) $, we have $ L(s_{0},\theta_{0};n) \to +\infty $ as $ n \to +\infty $, and so 
\[
\limsup_{n \to +\infty} \frac{L(s_{0},\theta_{0};n)}{n} = 1.
\]
Let $ A = \cos (\pi \lambda/2) $. From \eqref{eq:lambda-m}, it follows that
\[ 
\alpha_{n}(s_{0},\theta_{0}) \ge A \mu^{L(s_{0},\theta_{0};n)}
\]
for $ n $ sufficiently large. But $ \|D\Phi^{n}(s_{0},\theta_{0}) \| \ge \alpha_{n}(s_{0},\theta_{0}) $ so that 
\[
		\limsup_{n\to\infty} \frac{1}{n}\log\left\|D\Phi^n(s,\theta)\right\| \ge \log \mu > 0.
\]		
This and Proposition \ref{pr:dominated} imply that $ D \setminus B(\mathcal{P}) $ is hyperbolic. 
\end{proof}

\begin{remark}
	\label{re:uniform-vs-nonuniform}
	Suppose that $ j(n) $ in the proof of Proposition~\ref{pr:rectangle} is bounded from above by some constant $ J $ for every $ n $ and every $ (s_{0},\theta_{0}) \in D \setminus B(\mathcal{P}) $. Then $ D \setminus B(\mathcal{P}) $ is uniformly hyperbolic by Proposition~\ref{pr:delta-zero}. We can also obtain the same conclusion, in a more direct way, by arguing as follows. If such a $ J $ exists, then $ L(s_{0},\theta_{0}) \ge n-J $. The notation here is as in the proof of Proposition~\ref{pr:delta-zero}. We then obtain,
	\[ 
	\alpha_{n}(s_{0},\theta_{0}) \ge A \mu^{L(s_{0},\theta_{0};n)} \ge \frac{A}{\mu^{j}} \mu^{n},
	\]
	which means that $ D \setminus B(\mathcal{P}) $ is uniformly hyperbolic. 
\end{remark}

\begin{remark}
Since the estimate~\eqref{eq:lambda-m} relies only on the assumption that a rectangle has three consecutive sides with internal angles equal to $ \pi/2 $, one can easily extend Proposition~\ref{pr:rectangle} to polygons that are finite unions of rectangles with sides parallel to two orthogonal axes. 
\end{remark}

\begin{remark}
	\label{re:decreasing}
	It is not difficult to see that $ B(\mathcal{P}) = M \setminus N^{+}_{\infty} $ for every $ f \in \mathcal{R}^{k}_{1} $ with $ f'<0 $. This kind of reflection law is the contracting version of $ f(\theta) = -\theta $, the reflection law of Andreev billiards, which are models for the electronic conduction in a metal connected to a super-conductor \cite{Richter99}.
\end{remark}

We now give some sufficient conditions for the existence of a forward invariant set of positive $ \nu $-measure containing a uniformly hyperbolic set.
In view of Remark~\ref{re:decreasing}, we may assume that $ f'>0 $. Moreover, by rescaling $ P $, we may assume without loss of generality that the sides of $ P $ have length $ 1 $ and $ 0 < h \le 1 $. Fix a pair of parallel sides of $ P $. In the following, the expression `long sides' (resp. `long side') of $ P $ has to be understood as the chosen pair (resp. a side from the chosen pair) of parallel sides of $ P $ whenever $ h=1 $.

Define
\[ 
f_{1}(\theta) = - f(\theta) \quad \text{for } \theta \in [\pi/2,\pi/2],
\]
and
\[
f_{2}(\theta) = f\left(\sgn(\theta)\frac{\pi}{2}-\theta \right) \quad \text{for } \theta \in [-\pi/2,0) \cup (0,\pi/2].
\]
To explain the geometrical meaning of $ f_{1} $ and $ f_{2} $, we observe that if $ (s_{0},\theta_{0}) $ and $ (s_{1},\theta_{1}) $ are two consecutive collisions, then $ \theta_{1} = f_{1}(\theta_{0}) $ when the two collisions occur at parallel sides of $ P $, and $ \theta_{1}= f_{2}(\theta_{0}) $ when the two collisions occur at adjacent sides of $ P $. 

From the properties of $ f $, one can easily deduce that i) $ f_{2} $ is strictly decreasing, ii) the restriction of $ f_{2} $ to each interval $ (-\pi/2,0) $ and $ (0,\pi/2) $ is a strict contraction with Lipschitz constant less than $ \lambda(f) $, and iii) $ f_{2} $ has two fixed points $ \theta^{-} < 0 < \theta^{+} $.

\begin{defn}
Let $ T $ be the subset of $ M \setminus N^{+}_{\infty} $ consisting of elements whose positive semi-trajectory is trapped between two parallel sides of $ P $. 
\end{defn}

It is easy to see that $ T $ is a forward invariant set containing $ \PP $.

\begin{defn}
Let $ \hat{M} =  M \setminus (T \cup N^{+}_{\infty}) $. 
\end{defn} 

The set $ \hat{M} $ consists of elements with infinite positive semi-trajectory having at least a pair of consecutive collisions at adjacent sides of $ P $. It can be shown that $ \hat{M} $ is open up to a set of zero $ \nu $-measure. 

\begin{defn}
	The sequence $ \Phi(x_{0}),\ldots,\Phi^{n}(x_{0}) $ of $ n \ge 2 $ collisions with $ x_{0} \in \hat{M} $ is called \emph{adjacent} if $ \Phi^{i}(x_{0}) $ and $ \Phi^{i+1}(x_{0}) $ hit adjacent sides of $ P $ for every $ i = 1,\ldots, n-1 $. 
\end{defn}	

\begin{defn}	
	The sequence $ \Phi(x_{0}),\ldots,\Phi^{n-1}(x_{0}) $ of $ n \ge 1 $ collisions hitting two parallel sides of $ P $ with $ x_{0} \in \hat{M} $ is called \emph{complete} if $ x_{0} $ and $ x_{1} $ are adjacent, and $ x_{n-1} $ and $ x_{n} $ are adjacent. The number $ n $ is called the \emph{length} of the sequence.
\end{defn}

Let $ \theta^{-}_{*} < 0 <\theta^{+}_{*} $ be the two solutions of 
\[
\mathcal{F}(\theta):=\sum^{\infty}_{n=0} \tan \left(\left|f^{n}_{1}(\theta)\right|\right) = \frac{1}{h}, \quad \theta \in (-\pi/2,\pi/2).
\]
If the endpoints of a long side of $ P $ have arclength parameters $ s=0 $ and $ s=1 $, then the billiard particle leaving the endpoint $ s=0 $ (resp. $ s=1 $) with an angle $ \theta^{+}_{*} $ (resp. $ \theta^{-}_{*} $) bounces between the long sides of $ P $ forever, converging to the short side of $ P $ that has an endpoint at $ s=1 $ (resp. $ s=0 $). 

\begin{lemma}
	\label{le:trapped}
If $ (s_{0},\theta_{0}) \in M \setminus N^{+}_{\infty} $ is the first element of a complete sequence of collisions between the long sides of $ P $ and 
\begin{equation}
	\label{eq:cond-trapped}
f_{1}(\theta^{+}_{*}) \le \theta_{0} \le f_{1}(\theta^{-}_{*}),
\end{equation} 
then $ (s_{0},\theta_{0}) \in T $. 
\end{lemma}

\begin{proof}
	From the monotonicity of $ f $ and Condition~\eqref{eq:cond-trapped}, it follows that $ \mathcal{F}(\theta_{0})<1/h $. This implies immediately the wanted conclusion.
\end{proof}

\begin{lemma}
	\label{le:escape}
	If $ (s_{0},\theta_{0}) \in M \setminus N^{+}_{\infty} $ is the first element of a complete sequence of collisions between parallel sides of $ P $ and 
	\begin{equation}
		\label{eq:cond-escape}
	\theta_{0} < f_{1}(\theta^{+}_{*}) \quad \text{or }\quad f_{1}(\theta^{-}_{*}) < \theta_{0},
	\end{equation}
	then $ (s_{m},\theta_{m}) $ and $ (s_{m+1},\theta_{m+1}) $ are adjacent for some $ m \ge 0 $.
\end{lemma}

\begin{proof}
	Using the monotonicity of $ f $ and Condition~\eqref{eq:cond-escape}, we obtain $ \mathcal{F}(\theta_{0}) > 1/h > h $, which proves the proposition.
\end{proof}

Let
\[
\tilde{\theta} : = \min \left\{f_{2}\left(f\left(\frac{\pi}{2}\right)\right),-f_{2}\left(f\left(-\frac{\pi}{2}\right)\right) \right\}.
\]

\begin{lemma}
	\label{le:adjacent}
	Let $ (s_{0},\theta_{0}) \in \tilde{D} $. If $ (s_{0},\theta_{0}) $ and $ (s_{1},\theta_{1}) $ are adjacent collisions, then $ |\theta_{1}| \ge \tilde{\theta} $.
\end{lemma}

\begin{proof}
	We have $ |\theta_{1}| = |f_{2}(f(\bar{\theta}_{0}))| $, because $ \theta_{1} = f_{2}(\theta_{0}) $ and $ \theta_{0} = f(\bar{\theta}_{0}) $ with $ \bar{\theta}_{0} \in (-\pi/2,\pi/2)  $. But $ (f_{2} \circ f)'<0 $, since $ f'>0 $. Thus, the minimum of $ |f_{2} \circ f| $ is attained at $ -\pi/2 $ or $ \pi/2 $, and therefore it is equal to $ \tilde{\theta} $. 
\end{proof}

\begin{proposition}
\label{pr:uh-rect}
Suppose that $ f'>0 $ and
\begin{equation}
	\label{eq:non-trapped}
\tilde{\theta} > \max\{-\theta^{-}_{*},\theta^{+}_{*}\}.
\end{equation}
Then there exists $ m>0 $ such that for every $ (s_{0},\theta_{0}) \in \hat{M} $, the positive semi-orbit $ \{(s_{n},\theta_{n})\}_{n \ge 0} $ contains infinitely many complete sequences, and each complete sequence contained in $ \{(s_{n},\theta_{n})\}_{n \ge 0} $ has length less or equal than $ m $.
\end{proposition}

\begin{proof}
	We can assume without loss of generality that $ s=0 $ and $ s=1 $ correspond to the endpoints of a long side of $ P $, and $ s=1 $ and $ s=1+h $ correspond to the endpoints of the adjacent short side. By \eqref{eq:non-trapped}, we can apply Lemma~\ref{le:escape} to the four collisions $ (0,\tilde{\theta}),(1,-\tilde{\theta}),(1,\tilde{\theta}),(1,-\tilde{\theta}) $ with the angle measured with respect to the normal to the long side for the first two collisions and with respect to the normal of the short side for the last two collisions. We conclude that there exists $ n>0 $ such that the first pair of adjacent collisions in the semi-orbit of each of $ (0,\tilde{\theta}),(1,-\tilde{\theta}),(1,\tilde{\theta}),(1,-\tilde{\theta}) $ comes after a sequence of no more than $ n>0 $ consecutive collisions between two parallel sides of $ P $.
		
	Now, let $ x_{0} \in \hat{M} $. From the definition of $ \hat{M} $, there exists $ m>0 $ such that $ x_{m} $ and $ x_{m+1} $ are adjacent collisions. Hence $ \theta_{m+1} = f_{2}(\theta_{m}) $, and from Lemma~\ref{le:adjacent}, it follows that $ |\theta_{m+1}| \ge \tilde{\theta} $. The conclusion obtained in the previous paragraph implies that we can find $ 1 \le k \le n $ such that the collisions $ x_{m+k} $ and $ x_{m+k+1} $ are adjacent. To complete the proof, we just need to observe that $ n $ depends only on $ \tilde{\theta} $.
\end{proof}

\begin{corollary}
\label{co:uh-rect}
Under the hypotheses of Proposition~\ref{pr:uh-rect}, we have
\begin{enumerate}
	\item $ \hat{M} $ is a forward invariant set,
	\item every invariant subset of $ \hat{M} $ is uniformly hyperbolic,
	\item $ B(\mathcal{P}) = T $.
\end{enumerate}
\end{corollary}

\begin{proof}
	Proposition~\ref{pr:uh-rect} implies that $ \hat{M} $ is forward invariant and has Property (A). This gives Part~(1), and by using Proposition~\ref{pr:delta-zero}, Part~(2) as well. Part~(3) follows from the definition of $ \hat{M} $.
\end{proof}

\begin{remark}
	From Corollary~\ref{co:uh-rect}, it follows that $ \hat{M} $ contains a hyperbolic attractor $ \mathcal{H} $. This gives us the following dynamical picture: up to a set of zero $ \nu $-measure, $ M $ is the union (mod 0) of two positive $ \nu $-measure sets, the basin of attraction of $ \mathcal{P} $ and the basin of attraction of $ \mathcal{H} $. Although Theorem~\ref{thm bill p<a} is formulated for polygons with $ \mathcal{P} = \emptyset $, it is not difficult to see that we can apply it to $ \mathcal{H} $. The challenging problem here is to prove that for a rectangle (satisfying the hypotheses of the corollary) the hypothesis $ p(S^{+}_{m})<\alpha(\Phi^{m}) $ for some $ m>0 $ of Theorem~\ref{thm bill p<a} is satisfied. 
\end{remark}

\begin{remark}
By modifying properly the definition of $ \tilde{\theta} $ and Condition~\ref{eq:non-trapped}, it is possible to extend Proposition~\ref{pr:uh-rect} to a larger class of polygons with parallel sides. 
\end{remark}

Finally, we give sufficient conditions for having $ B(\mathcal{P}) = M \setminus N^{+}_{\infty} $. Recall that $ \theta^{-} $ and $ \theta^{+} $ are the fixed points of $ f_{2} $.

\begin{proposition}
	\label{pr:parabolic}
Suppose that $ f'> 0 $ and
\begin{gather} 
	h \le \tan \tilde{\theta} \label{eq:adjacent} \\
	f_{1}(\theta^{+}_{*}) < \theta^{-} \quad \text{and} \quad \theta^{+} < f_{1}(\theta^{-}_{*}). \label{eq:trapped} 
\end{gather}
	Then $ B(\mathcal{P}) = M \setminus N^{+}_{\infty} $.
\end{proposition}

\begin{proof}
	Let $ x_{0} \in M \setminus N^{+}_{\infty} $. As usual, we write $ x_{n} = (s_{n},\theta_{n}) = \Phi^{n}(x_{0}) $. If $ \theta_{0} = 0 $, then $ x_{0} \in \mathcal{P} $, and there is nothing to prove. Hence, we may assume that $ \theta_{0} \neq 0 $. In this case, arguing by contradiction, we show that there exists $ n>0 $ such that $ x_{n} \in T $. 

Suppose that $ x_{n} \notin T $ for every $ n \ge 0 $. Then two cases can occur: either there exists $ m>0 $ such that all collisions $ x_{m},x_{m+1},\ldots $ are adjacent, or there exist two sequences $ \{m_{k}\} $ and $ \{n_{k}\} $ of positive integers such that $ 0 \le m_{k} < n_{k} < m_{k+1} $ and $ x_{m_{k}},\ldots,x_{n_{k}} $ is a complete sequence of collisions between the long sides of $ P $ for every $ k \ge 0 $.

In the first case, we have 
\[ 
\theta_{n} = f^{n-m}_{2}(\theta_{m}) \quad \text{for } n \ge m. 
\]
By the properties of $ f_{2} $, it follows that either $ \theta_{n} \to \theta^{-} $ or $ \theta_{n} \to \theta^{+} $ as $ n \to +\infty $. Condition~\eqref{eq:trapped} guarantees the existence of $ n>m $ such that $ x_{n} $ is a collision at one of the long sides of $ P $, and $  f_{1}(\theta^{*}_{+}) < \theta_{n} < f_{1}(\theta^{-}_{*}) $. Lemma~\ref{le:trapped} implies that that $ x_{n} \in T $, obtaining a contradiction.
	
We now consider the second case. Using Lemma~\ref{le:adjacent}, we see that Condition~\eqref{eq:adjacent} implies that $ x_{n_{k}+1} $ and $ x_{n_{k}+2} $ are adjacent collisions. The same is true, as a direct consequence of the definition of $ m_{k} $ and $ n_{k} $, for the sequences $ x_{n_{k}},x_{n_{k}+1} $ and $ x_{n_{k}+2},\ldots,x_{m_{k+1}} $ (possibly with $ n_{k}+2 = m_{k+1} $). Hence
\[
\theta_{m_{k+1}} = f_{2}^{m_{k+1}-n_{k}}(\theta_{n_{k}}).
\]
If $ m_{k+1}-n_{k} $ is odd, then from the properties of $ f_{2} $, we deduce that either $ 0 < \theta_{m_{k+1}} < \theta^{-} $ or $ 0 < \theta_{m_{k+1}} < \theta^{+} $. This together with Condition~\eqref{eq:trapped} and Lemma~\ref{le:trapped} gives the contradiction $ x_{m_{k+1}} \in T $. Now, suppose that $ m_{k+1}-n_{k} $ even. The assumption $ x_{m_{k}} \notin T $ implies that $ \theta_{m_{k+1}} < f_{1}(\theta^{+}_{*}) $ or $  f_{1}(\theta^{-}_{*}) < \theta_{m_{k+1}} $. Using  Condition~\eqref{eq:trapped}, we see that $ \theta_{m_{k+1}} < \theta^{-} $ or $ \theta^{+} < \theta_{m_{k+1}} $. This and the fact that $ f^{2}_{2} $ is increasing imply that
\[
|\theta_{m_{k+1}}| \le |\theta_{n_{k}}|.
\]
Since $ \theta_{n_{k}} = f^{n_{k}-m_{k}}_{1}(\theta_{m_{k}}) $, we have
\[
|\theta_{m_{k+1}}| \le |f^{n_{k}-m_{k}}_{1}(\theta_{m_{k}})| \le \lambda^{n_{k}-m_{k}}|\theta_{m_{k}}|.
\]
From this, it follows that $ |\theta_{m_{k}}| \to 0 $ as $ k \to + \infty $. Thus there exists $ k>0 $ such that $ f_{1}(\theta^{+}_{*}) < \theta_{m_{k}} < f_{1}(\theta^{-}_{*}) $, and again using Lemma~\ref{le:trapped}, we obtain the contradiction $ x_{m_{k}} \in T $. This completes the proof.
\end{proof}

\begin{remark}
When $ f $ is odd, the hypotheses of Propositions~\ref{pr:uh-rect} and \ref{pr:parabolic} simplify significantly. Indeed, in this case, we have $ \theta^{-}_{*} = -\theta^{+}_{*} $, $ \theta^{-} = - \theta^{+} $ and $ \tilde{\theta} = f(\pi/2-f(\pi/2)) $.
\end{remark}

\begin{remark}
If $ f(\theta) = \sigma \theta $, then it is easy to show that Condition~\ref{eq:non-trapped} is satisfied for every $ 0 < \sigma < 1 $ and every $ 0 < h \le 1 $ such that $ 2/\pi < \sigma h $ (see also \cite[Proposition~V.9]{MDDGP12} for the square). Moreover, given $ 0 < \sigma < 1 $, Conditions~\eqref{eq:adjacent} and \eqref{eq:trapped} are satisfied for $ h $ sufficiently small.
\end{remark}

\appendix
%!TEX root = dissipative_main.tex
%%%%%%%%%%%
\section{Proofs of some technical results}
\label{ap:first}

\begin{proof}[Proof of Proposition~\ref{pr:prop-sing}]
Let $ B_{k} = [\tilde{s}_{k},\tilde{s}_{k+1}] \times (-\pi/2,\pi/2) $ for $ 1 \le k \le n $, and let $ \gamma_{jk} = q^{-1}_{1}(C_{j}) \cap B_{k} $ with $ 1 \le j < k $ or $ k+1 < j \le n $. Then we have
\[
S^{+}_{1} = \left(\bigcup^{n}_{\substack{k=1 \\ 1 \le j \le k-1}} \overline{\gamma_{jk}}\right) \cup  \left(\bigcup^{n}_{\substack{k=1 \\ k+2 \le j \le n}} \overline{\gamma_{jk}}\right).
\]
We show that the sets $ \overline{\gamma_{jk}} $ are indeed the curves $\Gamma_{1},\ldots,\Gamma_{m} $. Fix $ j $ and $ k $. We can always find a Cartesian coordinate system $ (X,Y) $ of $ \Rr^{2} $ such that the side $ L_{k} $ of $ P $ lies on the $ X $-axis, and has the same orientation as the $ X $-axis. If $ (s,\theta) \in \gamma_{jk} $, then
\[ 
\theta = \varphi_{jk}(s) := \arctan \left(\frac{\bar{s}-s}{l}\right),
\]
where $ l $ is the distance between $ C_{j} $ and $ L_{k} $, and $ \bar{s} \, $ is the abscissa of the vertex $ C_{j} $. It is easy to see that $ \gamma_{jk} $ is the graph of $ \varphi_{jk} $ over an open interval $ I_{jk} \subsetneq (\tilde{s}_{k},\tilde{s}_{k+1}) $ if there exist $ x \in \partial \gamma_{jk} \setminus V $ and a vertex $ C_{i} $ with $ i \neq j $ such that $ q_{1}(x) = C_{i} $. Otherwise, $ \gamma_{jk} $ is the graph of $ \varphi_{jk} $ over the entire interval $ (\tilde{s}_{k},\tilde{s}_{k+1}) $. The function $ \varphi_{jk} $ admits a unique smooth extension up to the boundary of $ I_{jk} $. If we denote by $ \bar{\varphi}_{jk} $ such an extension, then $ \overline{\gamma_{jk}} = \operatorname{graph}(\bar{\varphi}_{jk}) $. Moreover, 
\begin{equation}
	\label{eq:der}
\overline{\varphi}'_{jk}(s) = - \frac{l}{l^{2}+(\bar{s}-s)^{2}} < 0 \qquad \text{for } s \in \overline{I_{jk}}.
\end{equation}
It is not difficult to see from \eqref{eq:der} that if $ \bar{\varphi}_{ik}(s) = \bar{\varphi}_{jk}(s) $ for some $ s \in [\tilde{s}_{k},\tilde{s}_{k+1}] $ and some $ i \neq j $, then $ \bar{\varphi}'_{ik}(s) \neq \bar{\varphi}'_{jk}(s) $. This conclusion and previous observations prove Parts~(1)-(3) of the proposition.

Part~(4) of the proposition is proved by contradiction. Suppose that there exists $ x \in \gamma_{ik} \cap \gamma_{jk} $ for some $ i \neq j $. Then we would have simultaneously $ q_{1}(x) = C_{i} $ and $ q_{1}(x) = C_{j} $, implying that $ i = j $. This completes the proof.
\end{proof}

\begin{proof}[Proof of Proposition~\ref{pr:singminus}]
First of all, recall that $ S^{+}_{1} $ and $ S^{-}_{1} $ are diffeomorphic. Since $ DF(s,\theta) = \diag(1,-f'(\theta)) $, the curves forming $ S^{-}_{1} $ are strictly increasing if $ f'>0 $ and strictly decreasing if $ f'<0 $. 
\end{proof}

\begin{proof}[Proof of Proposition~\ref{pr:zero1}]
The Jacobian $ J(s_{0},\theta_{0}) $ of $ \Phi_{f} $ at $ (s_{0},\theta_{0}) \in D_{f} $ is given by 
\begin{align}
	\label{eq:jac}
	J(s_{0},\theta_{0}) = \left|\frac{\cos \theta_{0}}{\cos \bar{\theta}_{1}}\right| \cdot |f'(\bar{\theta}_{1})|,
	%\le \sup_{(-\pi/2,\pi/2)} \frac{\lambda(f)}{|\cos (\pi - \delta - \lambda(f) \theta)|},
\end{align}
where $ (s_{1},\theta_{1}) = \Phi_{f}(s_{0},\theta_{0}) $. By Formula~\eqref{eq:angle-after}, we have
\begin{equation}
	\label{eq:angl}
	\bar{\theta}_{1} = \pi - \delta(L_{0},L_{1}) - \theta_{0},
\end{equation}
where $ L_{0} $ and $ L_{1} $ are the sides of $ P $ containing $ q(s_{0},\theta_{0}) $ and $ q(s_{1},\theta_{1}) $, respectively. The hypothesis of the proposition implies that there exists $ a> 0 $ such that
\begin{equation}
	\label{eq:pos}
	\cos(\pi - \delta(L_{0},L_{1})) \ge a \qquad \text{for } (s_{0},\theta_{0}) \in D_{f}.
\end{equation}
From \eqref{eq:jac}-\eqref{eq:pos} and the fact that $ |\theta_{0}| \le \lambda(f) \pi/2 $, it follows that
\[ 
J(s_{0},\theta_{0}) \le \frac{\lambda(f)}{|\cos (\pi - \delta - \theta_{0})|} \xrightarrow[\lambda(f) \to 0^{+}]{} 0
\] 
uniformly in $ (s_{0},\theta_{0}) \in D_{f} $. Hence, there exists $ \bar{J}<1 $ such that for $ \lambda(f) $ sufficiently small, we have $ J(s_{0},\theta_{0}) \le \bar{J} $ for every $ (s_{0},\theta_{0}) \in D_{f} $. This easily easily implies the conclusion of the proposition.
\end{proof}

\begin{proof}[Proof of Proposition~\ref{pr:zero2}]
The proposition follows immediately from the fact that $ J(s_{0},\theta_{0}) < 1 $ for every $ (s_{0},\theta_{0}) \in D_{f} $, which in turn is a direct consequence of \eqref{eq:jac} and the hypothesis of the proposition.
\end{proof}

\label{ap:second}

We now establish some facts about Chebyshev polynomials that are used in Section~\ref{sec:regpoly} to analyze the slap map of odd sided regular polygons. The slopes of these maps are trigonometric algebraic numbers having Chebyshev polynomials as their minimal polynomials. We use the following lemma due to Gauss \cite[Lemma 3.17]{IS}.

\begin{lemma}
Given a  polynomial $F(x)\in\Z[x]$,
if $F(x)$ factors as $P(x) Q(x)$ in $\Q[x]$ it also factors
as $\overline{P}(x) \overline{Q}(x)$ in $\Z[x]$.
More precisely there is some constant $c\in\Q\setminus\{0\}$ such that
$\overline{P}(x)=c \,P(x)\in \Z[x]$ and $\overline{Q}(x)=c^ {-1} \,Q(x)\in \Z[x]$.
Hence, given a polynomial in $\Z[x]$, it is irreducible over $\Z[x]$ iff
it is irreducible over $\Q[x]$.
\end{lemma}
   
The {\em Chebyshev polynomials} of the first kind are defined recursively by $$T_0(x)=1,
\quad T_1(x)=x, $$ $$T_{n+1}(x)=2\,x\, T_n(x)-T_{n-1}(x)\;,$$ from where it follows easily,
by induction, that 
\begin{equation}\label{Chebyshev:relation}
T_n(\cos(\theta))=\cos(n\theta)\;. 
\end{equation} 

Chebyshev polynomials are related in~\cite{WZ} with the minimal polynomials $\Psi_n(x)$ of the trigonometric algebraic numbers of the form $\beta_n=\cos(\frac{2\pi}{n})$, which are real parts of unity roots. The minimal polynomial $\Psi_n(x)$ is considered there as a monic polynomial with rational coefficients, i.e., $\Psi_n(x)\in \Q[x]$. The relation established is the following:

\begin{theorem}[\cite{WZ}]\label{WZ:thm}
For any integer $n\geq 1$,
\begin{enumerate}
\item[(a)] if $n=2\,s+1$,\, then\;
 $\displaystyle T_{s+1}(x)-T_s(x) = 2^s\,\prod_{d\vert n} \Psi_d(x) $;\\
\item[(b)]  if $n=2\,s$,\, then\;
 $\displaystyle  T_{s+1}(x)-T_{s-1}(x) = 2^s\,\prod_{d\vert n} \Psi_d(x) $.
\end{enumerate}
\end{theorem}

This theorem allows a recursive calculation of the polynomials $\Psi_n(x)$. One obtains for
instance $\Psi_1(x)=x-1$, $\Psi_2(x)=x+1$, $\Psi_3(x)=x+\frac{1}{2}$, $\Psi_4(x)=x$ and
$\Psi_5(x)= x^2+\frac{1}{2}\,x-\frac{1}{4}$. Define now $\tilde{\Psi}_1(x)=\Psi_1(x)=x-1$,
and $\tilde{\Psi}_n(x)=2^{{\rm deg}(\Psi_n)}\, \Psi_n(x)$ for every $n\geq 2$. Denote by
$\phi(n)$ the Euler function, which counts the number of integers co-prime with $n$ in the
range $[1,n]$.
 
\begin{corollary} For $n>2$, 
$\tilde{\Psi}_n(x)$ is a polynomial in $\Z[x]$ of degree $\phi(n)/2$. Moreover,
 \begin{enumerate}
\item[(a)]  if $n=2\,s+1$,\, then\;
 $\displaystyle T_{s+1}(x)-T_s(x) =  \prod_{d\vert n} \tilde{\Psi}_d(x) $, and
\item[(b)]  if $n=2\,s$,\, then \;
 $\displaystyle  T_{s+1}(x)-T_{s-1}(x) =  \prod_{d\vert n} \tilde{\Psi}_d(x) $.
\end{enumerate}
\end{corollary}

\begin{proof}
The degree of the polynomial $\Psi_n(x)$, or $\tilde{\Psi}_n(x)$, was computed
 in~ \cite{DL} to be $\phi(n)/2$.
The proof goes by induction in $n$, using Theorem~\ref{WZ:thm} and Gauss' lemma.
Notice that ${\rm deg}(T_{s})=s$.
\end{proof}
 
From the previous computations we get, for instance, $\tilde{\Psi}_1(x)=x-1$,
$\tilde{\Psi}_2(x)=2x+2$, $\tilde{\Psi}_3(x)=2x+1$, $\tilde{\Psi}_4(x)=2 x$ and
$\tilde{\Psi}_5(x)=4x^ 2+2x-1$.
 
\begin{corollary}\label{Psin:coef}
For any $n\geq 5$, the minimal polynomial $\tilde{\Psi}_n$ has leading coefficient $2^ r$,
where $r=\phi(n)/2$, and unit constant term $\tilde{\Psi}_n(0)=\pm 1$. Moreover,
$r=\phi(n)/2\geq 2$ when $n\geq 7$.
\end{corollary}

\begin{proof}
By recursion, $T_n(0)=0$ if $n$ is odd, and $T_n(0)=\pm 1$\, if $n$ is even. Hence, if
$n=2s+1$, then $T_{s+1}(0)-T_{s}(0)=\pm 1$, if $n=2s$ with $s$ odd, then
$T_{s+1}(0)-T_{s-1}(0)=\pm 2$, while if $n=2s$ with $s$ even, then
$T_{s+1}(0)-T_{s-1}(0)=0$. Note now that $\tilde{\Psi}_2(0)=2$ and $\tilde{\Psi}_4(0)=0$. By
induction we can prove that $\tilde{\Psi}_n(0)=\pm 1$ for every $n\geq 5$. Finally, we have
$\phi(6)=2$, but using that $\phi(n)$ is multiplicative it follows that $2r=\phi(n)\geq 4$
for every $n\geq 7$.
\end{proof}

\section*{Acknowledgements}

The authors were supported by Funda\c c\~ao para a Ci\^encia e a Tecnologia through the Program POCI 2010 and the Project ``Randomness in Deterministic Dynamical Systems and Applications'' (PTDC-MAT-105448-2008). G.~Del Magno wish to express their thanks to M.~Lenci and R.~Markarian for stimulating conversations.

\bibliographystyle{plain}
%\addcontentsline{toc}{section}{References}
%%\bibliography{rfrncs}

\end{document}